\documentclass[A4,twoside,11pt,reqno]{article}
\usepackage{amssymb}
\usepackage{amsthm}
\usepackage[tbtags]{amsmath}
\usepackage{doc}
\usepackage{latexsym}
\usepackage{amscd}

\usepackage{graphics}
\usepackage{epsfig}
\usepackage{eucal}
\usepackage{mathrsfs}

\usepackage{fancyhdr}
\font\headd=cmr8

\begin{document}
\thispagestyle{plain}
\markboth{}{}
\small{\addtocounter{page}{0} \pagestyle{plain}
\newtheorem{theorem}{Theorem}[section]
\newtheorem{lemma}[theorem]{Lemma}
\newtheorem{corollary}[theorem]{Corollary}
\newtheorem{proposition}[theorem]{Proposition}

\theoremstyle{definition}
\newtheorem{definition}[theorem]{Definition}
\newtheorem{example}[theorem]{Example}

\theoremstyle{remark}
\newtheorem{remark}{Remark}
\newtheorem{notation}{Notation}

\vspace{0.2in}
\begin{center}
\noindent{\LARGE \bf  Representations and (2,3)-cohomology of Bol algebras with applications}
\end{center}
\begin{center}
\noindent{ A. Nourou Issa \\
D\'epartement de Math\'ematiques\\ Universit\'e d'Abomey-Calavi\\ 01 BP 4521 Cotonou, B\'enin\\
{\verb|woraniss@yahoo.fr|}
}
\end{center}
\vspace{0.15in}

{\footnotesize {\bf Abstract.} A representation theory for Bol algebras is proposed.
For a suitable $(2,3$)-cohomology theory for Bol algebras, we define
a ($2,3$)-coboundary with companion and next we define a ($2,3$)-cohomology group. Deformations of Bol algebras are investigated. In particular,
one-parameter infinitesimal deformations of Bol algebras are characterized in terms of
Bol algebras of deformation type and ($2,3$)-cocycles with coefficients in the adjoint
representation. The ($2,3$)-cohomology group is also applied to study abelian extensions of
Bol algebras.
}
}

\pagestyle{myheadings}
 \markboth{\headd  {\sc{Issa} } $~~~~~~~~~~~~~~~~~~~~~~~~~~~~~~~~~~~~~~~~~~~~~~~~~~~~~~~~~~~~~~~~~~~~~$}
 {\headd $~~~~~~~~~~~~~~~~~~~~~~~~~~~~~~~~~~~~~~~~~~~~~~~~~~~~~~~~~~~~~~${ \sc{Issa}}}

\section{Introduction}
Bol algebras were introduced in \cite{M1, SM} as infinitesimal objects associated with
local smooth Bol loops, in a study of the differential geometry of these loops. Just as
the case of relationships between Lie algebras and Lie groups, Bol algebras are tangent
algebras to local smooth Bol loops (see also \cite{MS, S} for an overview).
Some topics of the general theory of Bol algebras
were considered in \cite{HP, KZ, M2, Per}. Apart from Akivis algebras or Lie-Yamaguti
algebras (first called ``generalized Lie triple systems'' \cite{Y1}), Bol algebras
constitute an important class of so-called binary-ternary algebras. As Lie-Yamaguti algebras, Bol algebras also constitute a generalization of Lie triple systems.
\par
Deformation problems appear in various areas of mathematics and mathematical physics. The deformation theory found applications, e.g., in number theory, quantum mechanics (deformation quantization) with connections in pure mathematics, etc. The deformation of algebraic structures started with the works \cite{G1, G2}. Next
the deformation theory is extended to Lie algebras in the works \cite{NR1, NR2}. The 
representations and deformations of various algebraic structures were considered by
many researchers. It is well known that deformation theory of algebras has many
applications in mathematics and physics.
\par
A suitable cohomology of a given algebraic structure controls its deformations and
abelian extensions. The representation and cohomology theory of Lie triple systems 
was first considered in \cite{Y2, Y4} (see also \cite{Ha, HoP, KT} for other approaches
to this theory). In \cite{Zh} the cohomology of Lie triple systems was studied using the 
cohomology of Leibniz algebras. A cohomology theory for Lie-Yamaguti algebras was initiated 
in \cite{Y5} while in \cite{ZL} infinitesimal deformations and abelian extensions of
Lie-Yamaguti algebras were characterized by the $(2,3)$-cohomology group (here one
observes a generalization of results from \cite{Zh} in full analogy of a generalization
of Lie triple systems by Lie-Yamaguti algebras). Infinitesimal and formal deformations of Lie-Yamaguti algebras were also considered in \cite{LCM}.
\par
As mentioned above, Bol algebras constitute another generalization of Lie triple systems. Analyzing the set of defining identities of a Bol algebra (see Definition \ref{Def2.1}), one observes that a Bol algebra differentiates from a Lie triple system by an additional binary skew-symmetric operation satisfying a single one identity (see the axiom (B2) in Definition \ref{Def2.1}). This seemingly modest generalization leads to significant complexities when trying to extend results from Lie triple systems up to Bol algebras (e.g., unlike the case of Lie-Yamaguti algebras which are also a generalization of Lie triple systems, the ternary operation in Bol algebras is not a derivation).
A cohomology and deformation theory for Bol algebras constitutes, so far, a missing part in their general theory.
So our concern in this paper is to set up a study of a lower order cohomology theory for Bol algebras. However we observed that some peculiarities appear while dealing with a
suitable definition of a $(2,3)$-coboundary for Bol algebras (in fact we found essential
to define a coboundary with companion). We show that our cohomology of
Bol algebras is suitable in a study of deformations and abelian extensions of Bol algebras.
\par
The organization of the paper is as follows. In section 2 we define a representation and a suitable $(2,3)$-cohomology group for a Bol algebra and we point out that any representation of a Maltsev algebra induces a representation of the associated Bol algebra. In section 3
$(2,3)$-cocycles are used to investigate deformations of
Bol algebras. It turns out that the characterization of  $t$-parameter infinitesimal deformations in terms of
$(2,3)$-cocycles requires an additional notion of a ``Bol algebra of deformation type''
(a similar notion is introduced in \cite{ZL} in case of deformations of Lie-Yamaguti
algebras, and it seems that the notion of algebras of deformation type is essential in
the deformation theory of some types of algebras). The section 4 is devoted to a study 
of abelian extensions of Bol algebras. The interplay between abelian split extensions of Bol algebras and their representations or their $(2,3)$-cohomology groups is considered. It is observed that there is a one-to-one correspondence between equivalent classes of abelian extensions of a Bol algebra $B$ by its $B$-module $V$ and elements of the $(2,3)$-cohomology group.
\par
Throughout this paper, all vector spaces and algebras are finite-dimensional and are considered over a base field
$\mathbb{K}$ of characteristic zero.

\section{Representation and cohomology}
In this section, we recall basic facts regarding Bol algebras and next we present a 
representation and a $(2,3)$-cohomology theory for Bol algebras. In \cite{NB} a representation of Bol algebras is defined, but from this definition representations of Lie triple systems (as in \cite{Y2, Y4, Zh}) are not recovered; moreover, it seems that such a representation of Bol algebras could not be applied, e.g., in a study of abelian extensions of Bol
algebras.
\begin{definition}\label{Def2.1} (\cite{M1, SM}).
A Bol algebra is a triple $(B,*,[ , , ])$, where
$B$ is a vector space, ``$*$'': $B \times B \rightarrow B$ a bilinear map (the binary
operation on $B$), and ``$[ , , ]$'': $B \times B \times B \rightarrow B$ a
trilinear map (the ternary operation on $B$) such that
\par
(B01) $x*y = - y*x$,
\par
(B02) $[x,y,z] = - [y,x,z]$,
\par
(B1) ${\circlearrowleft}_{x,y,z} [x,y,z] = 0$,
\par
(B2) $[x,y, u*v] = [x,y,u]*v + u*[x,y,v] + [u,v,x*y] - uv*xy$,
\par
(B3) $[x,y,[u,v,w]] = [[x,y,u],v,w] + [u,[x,y,v],w] + [u,v,[x,y,w]]$\\
for all $u,v,w,x,y,z \in B$, where ${\circlearrowleft}_{x,y,z}$ denotes the sum 
over cyclic permutation of $x,y,z$ and $uv*xy$ means $(u*v)*(x*y)$.
\end{definition}
Observe that when $x*y=0$ in ($B,*,[ , , ]$), we gets a \textit{Lie triple system}
$(B, [ , , ])$ so one could think of a Bol algebra $(B,*,[ , , ])$ as a Lie triple
system $(B, [ , , ])$ with an additional skew-symmetric binary operation ``$*$'' on
$B$ such that the compatibility condition (B2) holds (this vision of a Bol algebra is
of prime significance throughout  the present paper; e.g. see the proof of Theorem 3.4
in section 3). The algebra $(B,*,[ , , ])$ will be denoted simply by $B$.
\par
The definition of a Bol algebra as above is the one of a \textit{left} Bol algebra
(see also \cite{HP, Per}). In a study of the tangent algebra of a local analytic loop
satisfying the right Bol identity, one considers \textit{right} Bol algebras
\cite{MS, S, SM} (see also \cite{HP, KZ}). From a left Bol algebra
($B,*,[ , , ]$), one defines a right Bol algebra structure on the vector space $B$
by the new operations
\par
$x \cdot y := - x*y$,
\par
$(z,x,y) := -[x,y,z]$. \\
Throughout this paper, by a Bol algebra we always mean a left Bol algebra.
\par
In \cite{KZ} the classification of all two-dimensional right Bol algebras over $\mathbb{R}$ is given. From this classification we obtain the following example of two-dimensional left Bol algebra.
\begin{example}\label{Ex2.2}
 Let ${\mathcal{B}_2}$ be a two-dimensional real vector space with basis $\{e_1, e_2\}$. Define on ${\mathcal{B}_2}$ the following bilinear and trilinear nonzero products:
 \par
 $e_1 * e_2 = - e_2$,
 \par
 $[e_1, e_2, e_1] = \lambda e_2$,
 \par
 $e_i * e_j = - e_j * e_i$, $[e_i, e_j, e_k] = - [e_j, e_i, e_k]$, $i \neq j$, \\
 where $\lambda \in \mathbb{R}$. Then $({\mathcal{B}_2},*, [, ,])$ is a left Bol algebra.
\end{example}
\par
Recall that a Maltsev algebra is a nonassociative algebra $(A, \cdot )$ satisfying $x \cdot y = - y \cdot x$ and $xy \cdot zx = (xy \cdot z) \cdot x + (yz \cdot x) \cdot x + (zx \cdot x) \cdot y$ for all $x,y,z$ in $A$.
\par
A representation theory for Maltsev algebras was developed in \cite{Y3}. The definition of the representation of a Maltsev algebra is reminded in the proof of Proposition \ref{Pr2.7} below. For now, we recall the construction of left Bol algebras from Maltsev algebras.
\begin{example}\label{Ex2.3}
 Given any Maltsev algebra $(M,*)$, one gets a left Bol algebra $(M,*,[, ,])$ if define on $(M,*)$ a ternary operation as $[x,y,z] := \frac{1}{3} (x*yz - y*xz + 2xy*z)$ \cite{M1} (see also in the proof of Proposition \ref{Pr2.7} below). Therefore, any alternative algebra gives rise to a Bol algebra.
 \end{example}
\begin{definition}\label{Def2.4}(\cite{M1, SM}). Let $B$ be a Bol algebra. A linear map
$\Pi : B \rightarrow B$ satisfying, $\forall u,v,w \in B$,
\par
$\Pi (u*v) = (\Pi u) * v + u * (\Pi v) + [u,v,\chi] - uv* \chi$,
\par
$\Pi ([u,v,w]) = [\Pi u,v,w] + [u, \Pi v, w] + [u,v,\Pi w]$\\
for some $\chi \in B$ is called a pseudoderivation  with companion $\chi$ of $(B,*,[ , , ])$.
\end{definition}
From the axioms (B2) and (B3) it is clearly seen that the operators $\delta (x,y)$
defined by $\delta (x,y)(z) := [x,y,z]$ are pseudoderivations with companion $x*y$
of the Bol algebra $B$.
\par
The notion of a pseudoderivation with companion allows an embedding of any Bol
algebra into some enveloping Lie algebra: if consider the vector space of all 
pseudoderivations with companion $Pder (B) := \{ (\Pi , \chi)\}$ of a given Bol
algebra $(B,*,[ , , ])$, then the external direct sum $\mathfrak{g} := B \dot{+} Pder (B)$
is a Lie algebra and $B$ is isomorphic to a subspace of $\mathfrak{g}$ (\cite{M1}; see
also the survey \cite{MS}). The notion of a pseudoderivation is extended to the one of
a \textit{hypoderivation with companion} in case of hyporeductive triple algebras
(\cite{Iss}) which are a kind of a generalization of Bol algebras. In accordance with
the theme of the present paper, the notion of pseudoderivation of Bol algebras as above
is also generalized in the presence of a representation of a given Bol algebra (see Definition \ref{Def2.12} below).
\begin{definition}\label{Def2.5}
 Let $B$ be a Bol algebra and $V$ a vector space. Let
$\rho : B \rightarrow End (V)$ be a linear map, where $End (V)$ is the vector space of
all endomorphisms of $V$, and $D, \theta : B \times B \rightarrow End (V)$ bilinear maps.
The triple $(\rho ,D, \theta)$ is called a representation of $B$ in $V$ (and then $V$ (or
$(V,\rho ,D, \theta)$) is called a $B$-module) if the following conditions holds:
\par
$\bullet$ (R1) $D(x_1 , x_2) + \theta (x_1 , x_2) - \theta (x_2 , x_1) = 0$;
\par
$\bullet$ (R21) $[D (x_1 , x_2), \rho (y_1)] = \rho ([x_1, x_2, y_1]) - \theta (y_1, x_1 * x_2)
+ \rho (x_1 * x_2)\rho (y_1)$;
\par
$\bullet$ (R22) $\theta (x_1, y_1 * y_2) = \rho (y_1) \theta (x_1, y_2) - \rho (y_2) \theta (x_1, y_1)
- {\big (} D(y_1 , y_2) -  \rho (y_1 * y_2) {\big )} \rho (x_1)$;
\par
$\bullet$ (R31) $[D (x_1 , x_2), D (y_1 , y_2)] = D ([x_1 , x_2, y_1], y_2) + D (y_1, [x_1 , x_2, y_2])$;
\par
$\bullet$ (R32) $[D (x_1 , x_2), \theta (y_1, y_2)] = \theta ([x_1 , x_2, y_1], y_2) 
+ \theta (y_1, [x_1 , x_2, y_2])$;
\par
$\bullet$ (R33) $\theta (x_1 , [y_1, y_2, y_3]) = \theta (y_2, y_3) \theta (x_1, y_1) 
- \theta (y_1, y_3) \theta (x_1, y_2) + D (y_1, y_2)\theta (x_1, y_3)$\\
for all $x_i,y_i \in B$.
\end{definition}
Because of (R1) the representation $(\rho ,D, \theta)$ could be also written as $(\rho , \theta)$. Observe that if
$\rho = 0$ and $x*y = 0$, $\forall x,y \in B$, then (R1), (R31), (R32),
and (R33) imply that $(V,\theta)$ is a representation of the Lie triple system 
$(B,[ , , ])$ (see \cite{Zh}).
\begin{example}\label{Ex2.6}
In case when $V=B$, there is a natural representation (called the
\textit{adjoint representation}) of the Bol algebra $B$, where the representation maps
$\rho$, $D$, and $\theta$ are given by 
\par
$\rho (u)(v) = u*v$, $D(u,v)(w) = [u,v,w]$, $\theta (u,v)(w) = [w,u,v]$\\
for all $u,v,w \in B$.
\end{example}
It is well-known (see \cite{M1, MS}) that any Maltsev algebra has a natural Bol algebra structure (the associated Bol algebra; see also Example \ref{Ex2.3} above). Therefore a class of nontrivial examples of $(\rho, D,\theta)$ for Bol algebras is constructed from the following result.
\begin{proposition}\label{Pr2.7}
Let $(M,*)$ be a Maltsev algebra and $\rho$ a representation of $(M,*)$. Then $\rho$ induces a representation of the Bol algebra $(M,*,[ , , ])$ associated with $(M,*)$.
\end{proposition}
\begin{proof}
 Given a Maltsev algebra $(M,*)$ with representation $\rho$, the following notations were introduced in \cite{Y6}, for an integer $k$:
 \par
 $[x,y,z]_k := x*yz - y*xz +kxy*z$,
 \par
 ${\theta}_k (x,y) := \rho (x) \rho (y) + k \rho (y)\rho (x) - \rho (x*y)$,
 \par
 $D_k (x,y) := [\rho (x),\rho (y)] + k \rho (x*y)$.\\
 That $\rho$ is a representation of the Maltsev algebra $(M,*)$ means that $\rho$ satisfies either of the following equivalent conditions (see Lemma 7.1 in \cite{Y3}).
 \begin{equation}\label{E:2.1}
  \rho (x*yz) - \rho (z)\{ \rho (x), \rho (y)\} + \rho (y)\{ \rho (x), \rho (z)\} = \{\rho (x),\rho (y*z)\} - \rho (z)\rho (x*y) + \rho (y)\rho (x*z),
 \end{equation}
 \begin{equation}\label{E:2.2}
  [D_1 (x,y) , \rho (z)] = \rho ([x,y,z]_1),
 \end{equation}
 where $\{ \rho (x), \rho (y)\}$ denotes the Jordan product
 $\rho (x)\rho (y) + \rho (y)\rho (x)$. In \cite{L} it is proved that $(M,*)$ is a Lie triple system with respect to the ternary composition $[x,y,z]_2$ while from \cite{M1} we know that $(M,*)$ is a Bol algebra with respect to the ternary composition $[x,y,z] := \frac{1}{3} [x,y,z]_2$. Then $(M,*,[, ,])$ is called the Bol algebra associated with
 $(M,*)$.
 \par
 Now we define
 \par
 $\theta (x,y) := \frac{1}{3} {\theta}_2 (x,y)$,
 \par
 $D(x,y) := \frac{1}{3} D_2 (x,y)$.\\
 We claim that $(\rho ,D, \theta)$ as defined above is a representation of the Bol algebra $(M,*,[, ,])$ associated with $(M,*)$.
 \par
 That $(\rho ,D, \theta)$ satisfies (R1), (R31), (R32), and (R33) follows from \cite{Y6}. We just need to check (R21) and (R22) for $(\rho ,D, \theta)$. We have
 \par
 $3([D(x,y),\rho (z)] - \rho ([x,y,z]) + \theta (z,x*y) - \rho (x*y)\rho (z))$
 \par
 $= [3D(x,y),\rho (z)] - \rho (3[x,y,z]) + 3\theta (z,x*y) - 3\rho (x*y)\rho (z)$
 \par
 $= [[\rho (x),\rho (y)], \rho (z)] + [\rho (x*y),\rho (z)] - \rho (x*yz) + \rho (y*xz) - \rho (xy*z)$
 \par
 $= [[\rho (x),\rho (y)] + \rho (x*y), \rho (z)] - \rho (x*yz - y*xz + xy*z)$
 \par
 $= [D_1 (x,y),\rho (z)] - \rho ([x,y,z]_1)$
 \par
 $= 0$ (by (\ref{E:2.2}))\\
 and so $(\rho ,D, \theta)$ verifies (R21). Likewise we have
 \par
 $3(\theta (x,y*z) - \rho (y)\theta (x,z) + \rho (z)\theta (x,y) + D(y,z)\rho (x) - \rho (y*z)\rho (x))$
 \par
 $= 3\theta (x,y*z) - \rho (y)(3\theta (x,z)) + \rho (z)(3\theta (x,y)) + 3D(y,z)\rho (x) - 3\rho (y*z)\rho (x)$
 \par
 $= {\theta}_2 (x,y*z) - \rho (y){\theta}_2 (x,z) + \rho (z){\theta}_2 (x,y) + D_2(y,z)\rho (x) - 3\rho (y*z)\rho (x)$
 \par
 $= \rho (x)\rho (y*z) + \rho (y*z)\rho (x) - \rho (x*yz)- \rho (y)\rho (x)\rho (z) - \rho (y)\rho (z)\rho (x) + \rho (y)\rho (x*z)$
 \par
 $+ \rho (z)\rho (x)\rho (y) + \rho (z)\rho (y)\rho (x) - \rho (z)\rho (x*y)$
 \par
 $= \{ \rho (x), \rho (y*z)\} - \rho (z)\rho (x*y) + \rho (y)\rho (x*z)$
 \par
 $- (\rho (x*yz) - \rho (z) \{ \rho (x), \rho (y)\} + \rho (y) \{ \rho (x), \rho (z)\})$
 \par
 $= 0$ (by (\ref{E:2.1}))\\
 and so $(\rho ,D, \theta)$ verifies (R22). This completes the proof.
\end{proof}
As a specific application of Proposition \ref{Pr2.7} above, we consider the following example.
\begin{example}\label{Ex2.8}
 Let $M$ be a $4$-dimensional vector space over
 $\mathbb{K}$ with basis $\{e_1, e_2, e_3, e_4 \}$. If define on $(M,*)$ the nonzero products as
 \par
 $e_1 * e_2 = - e_2$, $e_1 * e_3 = - e_3$, $e_1 * e_4 = e_4$,
 \par
 $e_2 * e_3 = 2e_4$,
 \par
 $e_i * e_j = - e_j * e_i$, $i \neq j$,\\
 then $(M,*)$ is a Maltsev algebra \cite{Sag} (since $char \; \mathbb{K} = 0$, this is the unique $4$-dimensional Maltsev algebra over $\mathbb{K}$). Let $M_0$ be the space spanned by $\{e_1, e_2\}$ and $V$ the one spanned by $\{e_3, e_4 \}$. It is clearly seen that $(M_0, *)$ is a subalgebra of $(M,*)$ and $(V,*)$ is an ideal of
 $(M,*)$. The $2$-dimensional Bol algebra $(M_0,*,[, ,])$ associated with $(M_0, *)$ is given by the following nonzero products:
\par
$e_1 * e_2 = - e_2$,
\par
$[e_1, e_2, e_1] = -e_2$,
\par
$e_i * e_j = - e_j * e_i$, $[e_i, e_j, e_k] = - [e_j, e_i, e_k]$, $i \neq j$.
\par
Now let $\rho$ be a representation of $(M_0,*)$ into the space $V$ induced by the adjoint representation of $(M,*)$. The application of Proposition \ref{Pr2.7} gives the representation of the Bol algebra $(M_0,*,[, ,])$ into the space $V$ as follows:
\par
$\rho (e_1)(e_3) = -e_3$, $\rho (e_1)(e_4) = e_4$,
\par
$\rho (e_2)(e_3) = 2e_4$, $\rho (e_2)(e_4) = 0$,
\par
$\theta (e_1, e_2)(e_k)=0= \theta (e_2,e_1)(e_k)$, $D(e_1, e_2) (e_k)=0$, $k=3,4$.
\end{example}
The result next gives a characterization of a representation of Bol algebras (we omit 
its proof since it is similar to the one of Theorem  \ref{Th3.4} in section 3).
\begin{proposition}\label{Pr2.9}
Let $(B,*,[ , , ])$ be a Bol algebra and $V$ a vector space. Let $\rho : B \rightarrow End (V)$ be a linear map, and
$D, \theta : B \times B \rightarrow End (V)$ bilinear maps. Then $(\rho ,D, \theta)$ is
a representation of $B$ in $V$ if and only if $(B \oplus V, \widetilde{*}, \{ , , \})$
is a Bol algebra, where
\par
$(x_1 + u_1) \widetilde{*} (x_2 + u_2) := x_1 * x_2 + \rho (x_1)(u_2) - \rho (x_2)(u_1)$,
\par
$\{ x_1 + u_1, x_2 + u_2, x_3 + u_3 \} := [x_1, x_2, x_3] + D(x_1, x_2)(u_3) 
- \theta (x_1, x_3)(u_2) + \theta (x_2, x_3)(u_1) $ \\
for all $x_i \in B$ and $u_i \in V$.
\end{proposition}
The Bol algebra $(B \oplus V, \widetilde{*}, \{ , , \})$ is called the \textit{semidirect product} of the Bol 
algebra $B$ with its $B$-module $V$. The following useful property of representation maps follows from Definition \ref{Def2.5}.
\begin{lemma}\label{Lem2.10}
Let $(\rho ,D, \theta)$ be a representation of a Bol algebra. Then
\begin{equation}\label{E:2.3}
 [{\Delta}_{D, \rho} (x_1 , x_2), {\Delta}_{D, \rho} (y_1 , y_2)] = 
 {\Delta}_{D, \rho} ([x_1 , x_2, y_1], y_2) + {\Delta}_{D, \rho} (y_1, [x_1 , x_2, y_2])
 - {\Delta}_{D, \rho} (y_1 * y_2, x_1 * x_2)
\end{equation}
for all $x_i, y_i \in B$, where the operator
${\Delta}_{D, \rho}$ is defined as
${\Delta}_{D, \rho} (u,v) := D(u, v) - \rho (u*v)$.
\end{lemma}
\begin{proof}
 From (R21) we have
 \par
 $[D (x_1 , x_2), \rho (y_1 * y_2)] = \rho ([x_1, x_2, y_1 * y_2]) - \theta (y_1 * y_2, x_1 * x_2)
+ \rho (x_1 * x_2)\rho (y_1 * y_2)$\\
i.e., by (B2),
\begin{equation}\label{E:2.4}
\begin{aligned}
&[D (x_1 , x_2), \rho (y_1 * y_2)] \\
&= \rho ([x_1, x_2, y_1] * y_2) + \rho (y_1 * [x_1, x_2, y_2]) + \rho ([y_1, y_2, x_1 * x_2]) \\
 &- \rho (y_1 y_2 * x_1  x_2) - \theta (y_1 * y_2, x_1 * x_2)
 + \rho (x_1 * x_2)\rho (y_1 * y_2).\\
\end{aligned}
\end{equation}
Likewise we have
\begin{equation}\label{E:2.5}
\begin{aligned}
&[D (y_1 , y_2), \rho (x_1 * x_2)] \\
&= \rho ([y_1, y_2, x_1] * x_2) + \rho (x_1 * [y_1, y_2, x_2]) + \rho ([x_1, x_2, y_1 * y_2]) \\
 &- \rho (x_1 x_2 * y_1  y_2) - \theta (x_1 * x_2, y_1 * y_2)
 + \rho (y_1 * y_2)\rho (x_1 * x_2).\\
\end{aligned}
\end{equation}
Subtracting memberwise (\ref{E:2.5}) from (\ref{E:2.4}), we get
\begin{equation}\label{E:2.6}
\begin{aligned}
 &[D (x_1 , x_2), \rho (y_1 * y_2)] + [\rho (x_1 * x_2), D (y_1, y_2)]\\
 &= [\rho (x_1 * x_2), \rho (y_1 * y_2)] + \rho ([x_1, x_2, y_1] * y_2) + \rho (y_1 * [x_1, x_2, y_2])
 - \rho (y_1 y_2 * x_1  x_2) \\
 &+ \{ \rho ([y_1, y_2, x_1 * x_2]) - \rho ([y_1, y_2, x_1] * x_2) - \rho (x_1 * [y_1, y_2, x_2])
 - \rho ([x_1, x_2, y_1 * y_2])\\
 &+ \rho (x_1 x_2 * y_1  y_2) \} \\
 &- \theta (y_1 * y_2, x_1 * x_2) + \theta (x_1 * x_2, y_1 * y_2)
\end{aligned}
 \end{equation}
(note that the expression in $\{ \cdots \}$ is zero by (B2)). Now observe that (R1) implies that
$D(u,v) = \theta (v,u) - \theta (u,v)$. Thus , in
(\ref{E:2.6}), we have
$- \theta (y_1 * y_2, x_1 * x_2) + \theta (x_1 * x_2, y_1 * y_2) = D (y_1 * y_2, x_1 * x_2)$ 
and so (\ref{E:2.6}) now reads as
\begin{equation}\label{E:2.7}
\begin{aligned}
 &[D (x_1 , x_2), \rho (y_1 * y_2)] + [\rho (x_1 * x_2), D (y_1, y_2)]\\
 &= [\rho (x_1 * x_2), \rho (y_1 * y_2)] + \rho ([x_1, x_2, y_1] * y_2) + \rho (y_1 * [x_1, x_2, y_2])
 - \rho (y_1 y_2 * x_1  x_2)\\
 &+ D (y_1 * y_2, x_1 * x_2).
\end{aligned}
\end{equation}
Subtracting memberwise (\ref{E:2.7}) from (R31) and next rearranging terms, we obtain (\ref{E:2.3}).
\end{proof}
\par
From (\ref{E:2.3}) we get that the vector space spanned by all ${\sum}_{i,j} {\Delta}_{D, \rho} (u_i , u_j)$ is a Lie subalgebra of $End (V)$.
\par
An $(n,n+1)$-cohomology group for Lie-Yamaguti algebras is constructed in \cite{Y5}. A similar
construction could be carried out over Bol algebras, but this remains an open problem. However, because of applications in a study
of deformations and abelian extensions of Bol algebras (see sections 3 and 4 below), we
restrict ourself to a construction of $(2,3)$-cohomology groups for this class of algebras.
\begin{definition}\label{Def2.11}
Let $B$ be a Bol algebra, $V$ a $B$-module and $(\rho, D, \theta)$
a representation of $B$. Let $C^2 (B,V)$ be the space of bilinear maps $\nu : B \times B \rightarrow V$
such that $\nu (x_1, x_2) = - \nu (x_2, x_1)$ and $C^3 (B,V)$ be the space of trilinear maps 
$\omega : B \times B \times B \rightarrow V$ such that $\omega (x_1, x_2, x_3) = - \omega (x_2, x_1, x_3)$
for all $x_i \in B$. The pair $(\nu, \omega) \in C^2 (B,V) \times C^3 (B,V)$ is called a
$(2,3)$-cocycle if, for all $x_i, y_i \in B$, the following conditions hold:
\par
$\bullet$ (CC1) ${\circlearrowleft}_{x_1, x_2, x_3} \omega (x_1, x_2, x_3) = 0$,
\par
$\bullet$ (CC2) $\omega (x_1, x_2, y_1 * y_2) + D(x_1, x_2) \nu (y_1, y_2)$
\par
\hspace{1.5cm}$= \omega (y_1, y_2, x_1 * x_2) + D(y_1, y_2) \nu (x_1, x_2)
+ \nu ([x_1, x_2, y_1], y_2) + \nu (y_1, [x_1, x_2, y_2])$
\par
\hspace{1.5cm}$+ \rho (y_1) \omega (x_1, x_2, y_2) - \rho (y_2) \omega (x_1, x_2, y_1)
+ \rho (x_1 * x_2) \nu (y_1, y_2) - \rho (y_1 * y_2) \nu (x_1, x_2)$ 
\par
\hspace{1.5cm}$- \nu (y_1 * y_2, x_1 * x_2)$,
\par
$\bullet$ (CC3) $\omega (x_1, x_2, [y_1, y_2, y_3]) + D(x_1, x_2) \omega (y_1, y_2, y_3)$
\par
\hspace{1.5cm}$= \omega ([x_1, x_2, y_1], y_2, y_3) + \omega (y_1, [x_1, x_2, y_2], y_3)
+ \omega (y_1, y_2, [x_1, x_2, y_3])$
\par
\hspace{1.5cm}$+ D(y_1, y_2) \omega (x_1, x_2, y_3) + \theta (y_2, y_3) \omega (x_1, x_2, y_1)
- \theta (y_1, y_3) \omega (x_1, x_2, y_2)$.
\end{definition}
The vector space spanned by all $(2,3)$-cocycles will be denoted by $Z^2 (B,V) \times Z^3 (B,V)$.
\par
The notion of pseudoderivation with companion as given above is slightly generalized as follows.
\begin{definition}\label{Def2.12}
Let $B$ be a Bol algebra, $V$ a $B$-module and $(\rho, D, \theta)$
a representation of $B$. A linear map $f:B \rightarrow V$ is called a pseudoderivation
with companion $\chi \in V$ of $B$ into $V$ with respect to the representation  $(\rho, D, \theta)$ 
if, for all $x_i \in B$,
\par
$f(x_1 * x_2) = \rho (x_1) f(x_2) - \rho (x_2) f(x_1) +  {\Delta}_{D, \rho} (x_1 , x_2)(\chi)$,
\par
$f([x_1, x_2, x_3]) = \theta (x_2, x_3) f(x_1) - \theta (x_1, x_3) f(x_2)
+ D(x_1, x_2) f(x_3)$.
\end{definition}
Observe that, in case of the adjoint representation of $B$, $f=\delta (x,y)$ is a pseudoderivation
with companion $x*y$ of the Bol algebra $B$ (as in Definition \ref{Def2.4}). Perhaps this notion of
pseudoderivation with respect to a representation could help to generalize some aspects of
the general theory of Bol algebras as initiated in \cite{M1, S, SM}.
\begin{definition}\label{Def2.13}
Let $B$ be a Bol algebra, $V$ a $B$-module and $(\rho, D, \theta)$
a representation of $B$. Then $(\nu, \omega) \in C^2 (B,V) \times C^3 (B,V)$ is called a 
$(2,3)$-coboundary if there is a pair $(f, \chi)$, where $f:B \rightarrow V$ is a linear
map and $\chi \in V$ such that
\par
$\bullet$ (BB1) $\nu (x_1, x_2) = \rho (x_1) f(x_2) - \rho (x_2) f(x_1) 
+  {\Delta}_{D, \rho} (x_1 , x_2)(\chi) - f(x_1 * x_2)$,
\par
$\bullet$ (BB2) $\omega (x_1, x_2, x_3) = \theta (x_2, x_3) f(x_1) - \theta (x_1, x_3) f(x_2)
+ D(x_1, x_2) f(x_3) - f([x_1, x_2, x_3])$\\
for all $x_i \in B$. The element $\chi \in V$ is called the companion of $(\nu, \omega)$.
\end{definition}
The vector space spanned by all $(2 ,3)$-coboundaries is denoted by $B^2 (B,V) \times B^3 (B,V)$. Observe
that $f$ is a pseudoderivation of $B$ into $V$ with companion $\chi$ if and only if
$(\nu, \omega) = (0,0)$. The definition of a $(2,3)$-coboundary with companion as above, as well as the operator ${\Delta}_{D, \rho}$ (see Lemma \ref{Lem2.10}), is justified by the following result (see also Proposition \ref{Pr4.5} in section 4).
\begin{proposition}\label{Pr2.14}
$B^2 (B,V) \times B^3 (B,V)$ is contained in
$Z^2 (B,V) \times Z^3 (B,V)$.
\end{proposition}
\begin{proof}
 Assume that we are given a $(\nu, \omega) \in C^2 (B,V) \times C^3 (B,V)$ satisfying (BB1)
 and (BB2). We need to check that $(\nu, \omega)$ also satisfies (CC1)-(CC3).
 \par
 For (CC1) we have
 \par
 ${\circlearrowleft}_{x_1, x_2, x_3} \omega (x_1, x_2, x_3) = {\circlearrowleft}_{x_1, x_2, x_3}
 [\underbrace{D(x_1 , x_2) + \theta (x_1 , x_2) - \theta (x_2 , x_1)}_{(R1)}] f(x_3)$
 \par
 \hspace{3.5cm}$- f(\underbrace{{\circlearrowleft}_{x_1, x_2, x_3} [x_1, x_2, x_3]}_{(B1)})$\\
 and so (CC1) holds for $(\nu, \omega)$ by (R1) and (B1).
 \par
 For (CC2) we have
 \par
 $\omega (x_1, x_2, y_1 * y_2) + D(x_1, x_2) \nu (y_1, y_2) - D(y_1, y_2) \nu (x_1, x_2)$
 \par
 $- \nu ([x_1, x_2, y_1], y_2) - \nu (y_1, [x_1, x_2, y_2]) 
 + \rho (y_2) \omega (x_1, x_2, y_1) - \rho (y_1) \omega (x_1, x_2, y_2)$
 \par
 $- \omega (y_1, y_2, x_1 * x_2) - \rho (x_1 * x_2) \nu (y_1, y_2) + \rho (y_1 * y_2) \nu (x_1, x_2)
 + \nu (y_1 * y_2, x_1 * x_2)$
 \par
 $=\underbrace{-f([x_1, x_2, y_1 * y_2]) + f([x_1, x_2, y_1 ]* y_2) + f(y_1 * [x_1, x_2, y_2])
 + f([y_1, y_2, x_1 * x_2])}_{(B2)}$
 \par
 $\underbrace{- f(y_1 y_2 * x_1 x_2)}$
 \par
 $\underbrace{- D(x_1, x_2)\rho (y_2) f(y_1) + \rho ([x_1, x_2, y_2])f(y_1) 
 - \theta (y_2, x_1 * x_2)f(y_1) + \rho (y_2) D(x_1, x_2)f(y_1)}_{(R21)}$
 \par
 $\underbrace{+ \rho (x_1 * x_2) \rho (y_2) f(y_1)}$
 \par
 $\underbrace{+ D(x_1, x_2)\rho (y_1) f(y_2) - \rho ([x_1, x_2, y_1])f(y_2)
 - \rho (y_1) D[x_1, x_2) f(y_2) + \theta (y_1, x_1 * x_2)f(y_2)}_{(R21)}$
 \par
 $\underbrace{- \rho (x_1 * x_2) \rho (y_1) f(y_2)}$
 \par
 $\underbrace{+ \theta (x_2, y_1 * y_2)f(x_1) + D(y_1, y_2)\rho (x_2) f(x_1)
 + \rho (y_2) \theta (x_2, y_1) f(x_1) - \rho (y_1) \theta (x_2, y_2) f(x_1)}_{(R22)}$
 \par
 $\underbrace{- \rho (y_1 * y_2) \rho (x_2) f(x_1)}$
 \par
 $\underbrace{- \theta (x_1, y_1 * y_2)f(x_2) - D(y_1, y_2)\rho (x_1) f(x_2)
 - \rho (y_2) \theta (x_1, y_1) f(x_2) + \rho (y_1) \theta (x_1, y_2) f(x_2)}_{(R22)}$
 \par
 $\underbrace{+ \rho (y_1 * y_2) \rho (x_1) f(x_2)}$
 \par
 $+ {\Big (}\underbrace{ [D(x_1, x_2), D(y_1, y_2)] + [\rho (y_1 * y_2), D(x_1, x_2)] 
 + [D(y_1, y_2), \rho (x_1 * x_2)] }_{(\ref{E:2.3})}$
 \par
 $\underbrace{- D([x_1, x_2, y_1], y_2) 
 + \rho ([x_1, x_2, y_1] * y_2) - D(y_1, [x_1, x_2, y_2])
 + \rho (y_1 * [x_1, x_2, y_2])}$
 \par
 $\underbrace{+ [\rho (x_1 * x_2), \rho (y_1 * y_2)] + D(y_1 * y_2, x_1 * x_2)
 - \rho (y_1 y_2 * x_1  x_2)} {\Big )} (\chi)$\\
 and so (CC2) holds for $(\nu, \omega)$ by (B2), (R21), (R22) and (\ref{E:2.3}).
 \par
 For (CC3) we have
 \par
 $\omega (x_1, x_2, [y_1, y_2, y_3]) - \omega ([x_1, x_2, y_1], y_2, y_3) 
 - \omega (y_1, [x_1, x_2, y_2], y_3) - \omega (y_1, y_2, [x_1, x_2, y_3])$
 \par
 $+  D(x_1, x_2) \omega (y_1, y_2, y_3) - \theta (y_2, y_3) \omega (x_1, x_2, y_1)
+ \theta (y_1, y_3) \omega (x_1, x_2, y_2) - D(y_1, y_2) \omega (x_1, x_2, y_3)$
\par
$= {\Big (} \underbrace{\theta (x_2, [y_1, y_2, y_3]) - \theta (y_2, y_3) \theta (x_2, y_1)
+ \theta (y_1, y_3) \theta (x_2, y_2) - D (y_1, y_2) \theta (x_2, y_3)}_{(R33)} {\Big )}
f(x_1)$
\par
$+ {\Big (} \underbrace{- \theta (x_1, [y_1, y_2, y_3]) + \theta (y_2, y_3) \theta (x_1, y_1)
- \theta (y_1, y_3) \theta (x_1, y_2) + D (y_1, y_2) \theta (x_1, y_3)}_{(R33)} {\Big )}
f(x_2)$
\par
$+ {\Big (} \underbrace{- \theta ([x_1, x_2, y_2], y_3) - \theta (y_2, [x_1, x_2, y_3])
+ [D (x_1, x_2), \theta (y_2, y_3)] }_{(R32)} {\Big )} f(y_1)$
\par
$+  {\Big (} \underbrace{ \theta ([x_1, x_2, y_1], y_3) + \theta (y_1, [x_1, x_2, y_3])
- [D (x_1, x_2), \theta (y_1, y_3)] }_{(R32)} {\Big )} f(y_2)$
\par
$+ {\Big (} \underbrace{- D ([x_1, x_2, y_1], y_2) - D (y_1, [x_1, x_2, y_2])
+ [D (x_1, x_2), D (y_1, y_2)] }_{(R31)} {\Big )} f(y_3)$
\par
$+ f(\underbrace{ - [x_1, x_2, [y_1, y_2, y_3]] +  [[x_1, x_2, y_1], y_2, y_3] 
+ [y_1, [x_1, x_2, y_2], y_3] + [y_1, y_2, [x_1, x_2, y_3]]}_{(B3)})$\\
and so (CC3) holds for $(\nu, \omega)$ by (B3), (R31), (R32) and (R33). This
completes the proof.
\end{proof}
Consistent with Proposition \ref{Pr2.14} above, we can now give the following definition.
\begin{definition}\label{Def2.15}
The $(2,3)$-cohomology group of a Bol algebra $B$ with coefficients
in the $B$-module $V$ is the quotient space $H^2 (B,V) \times H^3 (B,V)
:= (Z^2 (B,V) \times Z^3 (B,V)) / (B^2 (B,V) \times B^3 (B,V))$.
\end{definition}
\section{Deformations}

\subsection{Infinitesimal deformations.}
In this subsection we study infinitesimal deformations of Bol algebras. For this purpose
we need the notion of a ``Bol algebra of deformation type''.
\begin{definition}\label{Def3.1}
A Bol algebra of deformation type is a quadruple $(\mathcal{B}, \mu, \nu, \omega)$,
where $\mathcal{B}$ is a vector space, $\mu, \nu : \mathcal{B} \times \mathcal{B} \rightarrow \mathcal{B}$
are bilinear maps and $\omega : \mathcal{B} \times \mathcal{B} \times \mathcal{B} \rightarrow \mathcal{B}$
is a trilinear map such that
\par
(B01') $\nu (x_1, x_2) = - \nu (x_2, x_1)$,
\par
(B02') $\mu (x_1, x_2) = - \mu (x_2, x_1)$,
\par
(B03') $\omega (x_1, x_2, x_3) = - \omega (x_2, x_1, x_3)$,
\par
(B1') ${\circlearrowleft}_{x_1, x_2, x_3}  \omega (x_1, x_2, x_3) = 0$,
\par
(B2') $\omega (x_1, x_2, \nu (y_1, y_2)) = \nu (\omega (x_1, x_2, y_1), y_2)
+ \nu (y_1, \omega (x_1, x_2, y_2)) + \omega (y_1, y_2, \nu (x_1, x_2))$
\par
\hspace{4.0cm}$- \nu (\nu (y_1, y_2), \mu (x_1, x_2)) - \nu (\mu (y_1, y_2), \nu (x_1, x_2))$
\par
\hspace{4.0cm}$-\mu (\nu (y_1, y_2),\nu (x_1, x_2))$
\par
(B3') $\omega (x_1, x_2, \omega (y_1, y_2, y_3)) = \omega (\omega (x_1, x_2, y_1), y_2, y_3)
+ \omega (y_1 , \omega (x_1, x_2, y_2), y_3)$
\par
\hspace{4.5cm}$+ \omega (y_1, y_2, \omega (x_1, x_2, y_3))$ \\
for all $x_i, y_i \in \mathcal{B}$. Then the triple $(\mu, \nu, \omega)$ is said to define a Bol algebra of deformation type.
\end{definition}
The conditions (B03'), (B1'), and (B3') say that $(\mathcal{B},\omega)$ is a Lie triple system. Therefore a Bol algebra of deformation type is a quadruple $(\mathcal{B}, \mu, \nu, \omega)$ such that $(\mathcal{B}, \omega)$ is a Lie triple system and (B01'), (B02'), (B2') hold. The condition (B2') could be seen as a compatibility condition.
\par
Let $B$ be a Bol algebra, $\nu : B \times B \rightarrow B$ and 
$\omega : B \times B \times B \rightarrow B$ be bilinear and trilinear maps respectively.
Consider a $t$-parametrized family of bilinear and trilinear operations:
\par
$x_1 *_t x_2 := x_1 * x_2 + t \nu (x_1, x_2)$,
\par
$[x_1, x_2, x_3]_t := [x_1, x_2, x_3] + t \omega (x_1, x_2, x_3)$,\\
$x_i \in B$.
\begin{definition}\label{Def3.2}
The pair $(\nu, \omega)$ is said to generate a $t$-parameter
infinitesimal deformation of $B$ if $B_t := (B, *_t, [ , , ]_t)$ is a Bol algebra.
\end{definition}
We can now give the following characterization of $t$-parameter infinitesimal
deformations of Bol algebras.
\begin{proposition}\label{Pr3.3}
With the notations above, $(\nu, \omega)$ generates
a $t$-parameter infinitesimal deformation of a Bol algebra $B$ if and only if
\par
(i) $(*,\nu, \omega)$ defines a Bol algebra of deformation type;
\par
(ii) $(\nu, \omega)$ is a $(2,3)$-cocycle of $B$ with coefficients in the adjoint
representation.
\end{proposition}
\begin{proof}
 Suppose that $(\nu, \omega)$ generates a $t$-parameter infinitesimal deformation of
 the Bol algebra $B$, i.e. $B_t$ is a Bol algebra. It is easy to see that the axioms
 (B01), (B02), and (B1) for $B_t$ are equivalent to
 \begin{equation}\label{E:3.1}
 \nu (x_1, x_2) = - \nu (x_2, x_1), 
 \end{equation}
 \begin{equation}\label{E:3.2}
 \omega (x_1, x_2, x_3) = - \omega (x_2, x_1, x_3), 
 \end{equation}
 \begin{equation}\label{E:3.3}
  {\circlearrowleft}_{x_1, x_2, x_3}  \omega (x_1, x_2, x_3) = 0
 \end{equation}
respectively. The left-hand side of (B2) for $B_t$ is 
\par
$[x_1, x_2, y_1 *_t y_ 2]_t = [x_1, x_2, y_1 * y_2]
+ t \big{(} [x_1, x_2, \nu (y_1, y_2)] + \omega (x_1, x_2, y_1 * y_2) \big{)}$
\par
\hspace{3.0cm}$+ t^2 \omega (x_1, x_2, \nu (y_1, y_2))$\\
while the computation of its right-hand side gives
\par
$[x_1, x_2, y_1]_t *_t y_2 + y_1 *_t [x_1, x_2, y_2]_t + [y_1, y_2, x_1 *_t x_2]_t
- (y_1 *_t y_2) *_t (x_1 *_t x_2)$
\par
$= [x_1, x_2, y_1]* y_2 + y_1 *[x_1, x_2, y_2] + [y_1, y_2, x_1 * x_2] - (y_1 * y_2)*(x_1 * x_2)$
\par
$+ t \{ \omega (x_1, x_2, y_1) * y_2 + y_1 * \omega (x_1, x_2, y_2) + \omega (y_1, y_2, x_1 * x_2)
+ [y_1, y_2, \nu (x_1, x_2)]$
\par
\hspace{0.5cm}$+ \nu ([x_1, x_2, y_1], y_2) + \nu (y_1, [x_1, x_2, y_2]) - (y_1 * y_2) * \nu (x_1, x_2)
- \nu (y_1, y_2) * (x_1 * x_2)$
\par
\hspace{0.5cm}$- \nu (y_1 * y_2,  x_1 * x_2) \} $
\par
$+ t^2 \{ \nu (\omega (x_1, x_2, y_1), y_2) + \nu (y_1, \omega (x_1, x_2, y_2)) 
 + \omega (y_1, y_2, \nu (x_1, x_2)) - \nu (\nu (y_1, y_2), x_1 * x_2)$
 \par
 \hspace{0.5cm}$- \nu (y_1 * y_2, \nu (x_1, x_2)) - \nu (y_1, y_2) * \nu (x_1, x_2) \}$
 \par
 $-o(t^3)$,\\
 where
 $o(t^3):= t^3 \nu (\nu (y_1, y_2), \nu (x_1, x_2))$ is the term in power $t^3$.
 Therefore the axiom (B2) for $B_t$ holds if and only if
 \begin{equation}\label{E:3.4}
  \begin{aligned}
&[x_1, x_2, \nu (y_1, y_2)] + \omega (x_1, x_2, y_1 * y_2) \\
&= \nu ([x_1, x_2, y_1], y_2) + \nu (y_1, [x_1, x_2, y_2]) + [y_1, y_2, \nu (x_1, x_2)]\\
&+ \omega (x_1, x_2, y_1) * y_2 + y_1 * \omega (x_1, x_2, y_2) + \omega (y_1, y_2, x_1 * x_2)\\
&- (y_1 * y_2) * \nu (x_1, x_2)
- \nu (y_1, y_2) * ( x_1 * x_2) - \nu (y_1 * y_2,  x_1 * x_2),
 \end{aligned}
 \end{equation}
 \begin{equation}\label{E:3.5}
  \begin{aligned}
  \omega (x_1, x_2, \nu (y_1, y_2)) &= \nu (\omega (x_1, x_2, y_1), y_2) + \nu (y_1, \omega (x_1, x_2, y_2))\\
  &+ \omega (y_1, y_2, \nu (x_1, x_2)) - \nu (\nu (y_1, y_2), x_1 * x_2)\\
  &- \nu (y_1 * y_2, \nu (x_1, x_2)) - \nu (y_1, y_2) * \nu (x_1, x_2).
  \end{aligned}
 \end{equation}
Finally, proceeding as for (B2) above, one finds that the axiom (B3) for $B_t$ holds if and only if
\begin{equation}\label{E:3.7}
 \begin{aligned}
  &[x_1, x_2, \omega (y_1, y_2, y_3)] + \omega (x_1, x_2, [y_1, y_2, y_3])\\
  &= \omega ([x_1, x_2, y_1], y_2, y_3) + \omega (y_1 , [x_1, x_2, y_2], y_3) + \omega (y_1, y_2, [x_1, x_2, y_3])\\
  &+ [\omega (x_1, x_2, y_1), y_2, y_3] + [y_1 , \omega (x_1, x_2, y_2), y_3] + [y_1, y_2, \omega (x_1, x_2, y_3)],
 \end{aligned}
\end{equation}
and
\begin{equation}\label{E:3.8}
 \omega (x_1, x_2, \omega (y_1, y_2, y_3)) = \omega (\omega (x_1, x_2, y_1), y_2, y_3)
+ \omega (y_1 , \omega (x_1, x_2, y_2), y_3) + \omega (y_1, y_2, \omega (x_1, x_2, y_3)).
\end{equation}
Thus the proof is complete by (\ref{E:3.1})-(\ref{E:3.8}) since (\ref{E:3.1}), (\ref{E:3.2}),
(\ref{E:3.3}), (\ref{E:3.5}) and (\ref{E:3.8}) define a Bol algebra of deformation type while (\ref{E:3.3}), (\ref{E:3.4}) and
(\ref{E:3.7}) define a $(2,3)$-cocycle of $B$ with coefficients in the adjoint representation.
\end{proof}
\par
The following result shows that, given a Bol algebra $B$, any $(2,3)$-cocycle of $B$ as in Definition \ref{Def2.11} gives rise to a Bol structure on $B \oplus V$.
\begin{theorem}\label{Th3.4}
Let $B$ be a Bol algebra, $(V, \rho, D, \theta)$ a $B$-module and $(\nu, \omega)$
a $(2,3)$-cocycle of $B$. If define on $B \oplus V$ the operations
\begin{equation}\label{E:3.9}
 (x_1 + u_1) *_{\nu} (x_2 + u_2) := x_1 * x_2 + \nu (x_1, x_2) + \rho (x_1)(u_2) - \rho (x_2)(u_1)
\end{equation}
\begin{eqnarray}\label{E:3.10}
 [x_1 + u_1, x_2 + u_2, x_3 + u_3]_{\omega} &:=& [x_1, x_2, x_3] + \omega (x_1, x_2, x_3) + D(x_1, x_2)(u_3) \\ \nonumber
 & & -\theta (x_1, x_3)(u_2) + \theta (x_2, x_3)(u_1),
 \end{eqnarray}
where $x_i \in B$ and $u_i \in V$, then $B {\oplus}_{(\nu, \omega)} V :=
( B \oplus V, *_{\nu} , [ , , ]_{\omega})$ is a Bol algebra.
\end{theorem}
\begin{proof}
 That $( B \oplus V, [ , , ]_{\omega})$ is a Lie triple system follows from Lemma 4.5 in \cite{Zh}. Thus we are done
 if we check (B2) for $B {\oplus}_{(\nu, \omega)} V$. We have
 \par
 $[x_1 + u_1, x_2 + u_2, (y_1 + v_1) *_{\nu} (y_2 + v_2)]_{\omega}$
 \par
 $= \underbrace{[x_1, x_2, y_1 * y_2]}_{(B2)} 
 + \underbrace{\omega (x_1, x_2, y_1 * y_2) + D(x_1, x_2) \nu (y_1, y_2)}_{(CC2)}$
 \par
 $+ \underbrace{\theta (x_2, y_1 * y_2) (u_1) - \theta (x_1, y_1 * y_2) (u_2)}_{(R22)}
 \underbrace{- D(x_1, x_2) \rho (y_2)(v_1) + D(x_1, x_2) \rho (y_1)(v_2)}_{(R21)}$\\
 while
 \par
 $[x_1 + u_1, x_2 + u_2, y_1 + v_1]_{\omega} *_{\nu} (y_2 + v_2) + (y_1 + v_1) *_{\nu}
 [x_1 + u_1, x_2 + u_2, y_2 + v_2]_{\omega}$
 \par
 $+ [y_1 + v_1, y_2 + v_2, (x_1 + u_1) *_{\nu} (x_2 + u_2)]_{\omega}$
 \par
 $- ((y_1 + v_1) *_{\nu} (y_2 + v_2)) *_{\nu} ((x_1 + u_1) *_{\nu} (x_2 + u_2))$
 \par
 $= \underbrace{[x_1, x_2, y_1]* y_2 + y_1 * [x_1, x_2, y_2] + [y_1, y_2, x_1 * x_2] - (y_1 * y_2) * (x_1 * x_2)}_{(B2)}$
 \par
 $+ \underbrace{ \nu ([x_1, x_2, y_1], y_2) + \nu (y_1, [x_1, x_2, y_2])-\rho (y_2) \omega (x_1, x_2, y_1)}_{(CC2)}$
 \par
 $+ \underbrace{\rho (y_1) \omega (x_1, x_2, y_2) + \omega (y_1, y_2, x_1 * x_2) + D(y_1, y_2) \nu (x_1, x_2)}_{(CC2)}$
 \par
 $\underbrace{- \nu (y_1 * y_2, x_1 * x_2) - \rho (y_1 * y_2) \nu (x_1, x_2) + \rho (x_1 * x_2) \nu (y_1, y_2)}_{(CC2)}$
 \par
 $+ \Big{(} \underbrace{\rho (y_1) \theta (x_2, y_2) - \rho (y_2) \theta (x_2, y_1)
 - {\Delta}_{D, \rho} (y_1 , y_2) \rho (x_2)}_{(R22)} \Big{)} (u_1)$
 \par
 $+ \Big{(} \underbrace{\rho (y_2) \theta (x_1, y_1) - \rho (y_1) \theta (x_1, y_2)
 + {\Delta}_{D, \rho} (y_1 , y_2) \rho (x_1)}_{(R22)} \Big{)} (u_2)$
 \par
 $+ \Big{(} \underbrace{-\rho (y_2) D(x_1, x_2) - \rho ([x_1, x_2, y_2] +  \theta (y_2, x_1 * x_2)
 - \rho (x_1 * x_2) \rho (y_2)}_{(R21)} \Big{)} (v_1)$
 \par
 $+ \Big{(} \underbrace{\rho (y_1) D(x_1, x_2) + \rho ([x_1, x_2, y_1] -  \theta (y_1, x_1 * x_2)
 + \rho (x_1 * x_2) \rho (y_1)}_{(R21)} \Big{)} (v_2)$.\\
 Therefore (B2) for $B {\oplus}_{(\nu, \omega)} V$ holds by (B2) for $B$, (CC2), (R21) and (R22).
\end{proof}
The proof of Theorem \ref{Th3.4} above hints a proof of Proposition \ref{Pr2.9} in section 2.
\subsection{One-parameter formal deformations.}
Let $(B,*,[, ,])$ be a Bol algebra  and $\mathbb{K}[[t]]$ the ring of power series in one variable $t$ with coefficients in
$\mathbb{K}$. Denote by $B[[t]]$ the set of power series with coefficients in the vector space
$B$.
\begin{definition}\label{Def3.5}
A $1$-parameter formal deformation of the Bol algebra $B$ is given
by a pair $(f_{t}, g_{t})$, where $f_{t}$ is a $\mathbb{K} [[t]]$-bilinear map $f_{t} : B[[t]] \times B[[t]] \rightarrow B[[t]]$,
$f_{t} := {\sum}_{i \geq 0} F_{i} t^{i}$,  and $g_{t}$ a
$\mathbb{K} [[t]]$-trilinear map $g_{t} : B[[t]] \times
B[[t]] \times B[[t]] \rightarrow B[[t]]$, $g_{t} := {\sum}_{i \geq 0} G_{i} t^{i}$, where $F_{0} := *$, $G_{0} := [ , ,]$,
and each $F_{i}$ (resp. $G_{i}$)  is a $\mathbb{K}$-bilinear map $B \times B \rightarrow B$ extended to be
$\mathbb{K} [[t]]$-bilinear (resp. $\mathbb{K}$-trilinear map $B \times B \times B \rightarrow B$ extended to be
$\mathbb{K} [[t]]$-trilinear), such that $B_{T} := (B[[t]], f_{t}, g_{t})$ is a Bol algebra over $\mathbb{K} [[t]]$. A deformation  $(f_{t}, g_{t})$ is said to be of order $l$ whenever $f_{t} := {\sum}_{i=0}^{l} F_{i} t^{i}$,
$g_{t} := {\sum}_{i=0}^{l} G_{i} t^{i}$.
\end{definition}
That $B_T$ is a Bol algebra means that, for all $x_i$, $y_i$ $\in B$, \\
\begin{equation}\label{E:3.11}
 f_t (x_1 , x_2) = - f_t (x_2 , x_1),
\end{equation}
\begin{equation}\label{E:3.12}
 g_t (x_1 , x_2, x_3) = - g_t (x_2 , x_1, x_3),
\end{equation}
\begin{equation}\label{E:3.13}
 {\circlearrowleft}_{x_1, x_2, x_3}  g_t (x_1 , x_2, x_3) = 0,
\end{equation}
\begin{equation}\label{E:3.14}
\begin{aligned}
 g_t (x_1 , x_2, f_t (y_1 , y_2)) =  &f_t (g_t (x_1 , x_2, y_1), y_2) + f_t (y_1, g_t (x_1 , x_2, y_2))\\
 &+ g_t (y_1 , y_2, f_t (x_1 , x_2))
 - f_t (f_t (y_1 , y_2, f_t (x_1 ,x_2)),
 \end{aligned}
 \end{equation}
\begin{equation}\label{E:3.15}
\begin{aligned}
 g_t (x_1 , x_2, g_t (y_1 , y_2, y_3)) =  &g_t (g_t (x_1 , x_2, y_1), y_2, y_3) + g_t (y_1, g_t (x_1 , x_2, y_2),y_3)\\
 &+ g_t (y_1 , y_2, g_t (x_1 , x_2, y_3)).
 \end{aligned}
 \end{equation}
The relations (\ref{E:3.11})-(\ref{E:3.15}) are equivalent to the following set of equations (which are called the deformation equations of a Bol algebra):\\
\begin{equation}\label{E:3.16}
 F_k (x_1 , x_2) = - F_k (x_2 , x_1),
\end{equation}
\begin{equation}\label{E:3.17}
 G_k (x_1 , x_2, x_3) = - G_k (x_2 , x_1, x_3),
\end{equation}
\begin{equation}\label{E:3.18}
 {\circlearrowleft}_{x_1, x_2, x_3}  G_k (x_1 , x_2, x_3) = 0,
\end{equation}
\begin{equation}\label{E:3.19}
\begin{aligned}
 {\sum}_{i=0}^{k} G_i (x_1, x_2, F_{k-i}(y_1,y_2)) &= {\sum}_{i=0}^{k} {\Big (} F_i (G_{k-i} (x_1, x_2, y_1),y_2) + F_i (y_1, G_{k-i} (x_1, x_2, y_2))\\
 &+ G_i (y_1, y_2, F_{k-i} (x_1, x_2)) {\Big )} -
 {\sum}_{r+s+t} F_r (F_s (y_1, y_2),F_t (x_1, x_2)),
 \end{aligned}
\end{equation}
\begin{equation}\label{E:3.20}
\begin{aligned}
 {\sum}_{i=0}^{k} G_i (x_1, x_2, G_{k-i}(y_1,y_2,y_3)) &= {\sum}_{i=0}^{k} {\Big (} G_i (G_{k-i} (x_1, x_2, y_1),y_2,y_3)\\
 &+ G_i (y_1, G_{k-i} (x_1, x_2, y_2),y_3)\\
 &+ G_i (y_1, y_2, G_{k-i} (x_1, x_2, y_3)){\Big )},
 \end{aligned}
 \end{equation}
for $k \geq 0$.
\par
The equations (\ref{E:3.16}) and (\ref{E:3.17}) imply that
$(F_k , G_k) \in C^2 (B,V) \times C^3 (B,V)$. We will be interested in formal deformations $(f_t, g_t)$ of first order i.e. $f_t (a,b) = F_0 (a,b) + F_1 (a,b)t$, $g_t (a,b,c) = G_0 (a,b,c) + G_1 (a,b,c)t$ and we shall consider the set (\ref{E:3.16})-(\ref{E:3.20}) when $k=1,2,3$.
\par
For $k=1$, (\ref{E:3.18})-(\ref{E:3.20}) lead to
\begin{equation}\label{E:3.21}
 {\circlearrowleft}_{x_1, x_2, x_3}  G_1 (x_1 , x_2, x_3) = 0,
\end{equation}
\begin{equation}\nonumber
 \begin{aligned}
&G_0 (x_1, x_2, F_1 (y_1,y_2)) + G_1 (x_1, x_2, F_0 (y_1,y_2)) = F_0 (G_1 (x_1, x_2, y_1), y_2) + F_0 (y_1, G_1 (x_1, x_2, y_2))\\
&+ G_0 (y_1, y_2, F_1 (x_1,x_2)) - F_0 (F_0 (y_1, y_2), F_1 (x_1, x_2)) - F_0 (F_1 (y_1, y_2), F_0 (x_1, x_2))\\
&+ F_1 (G_0 (x_1, x_2, y_1), y_2) + F_1 (y_1, G_0 (x_1, x_2, y_2) + G_1 (y_1, y_2, F_0 (x_1,x_2))\\
&- F_1 (F_0 (y_1, y_2), F_0 (x_1,x_2))
\end{aligned}
\end{equation}
i.e.
\begin{equation}\label{E:3.22}
\begin{aligned}
&[x_1, x_2, F_1 (y_1, y_2)] + G_1 (x_1, x_2, y_1 *y_2)
= G_1 (x_1, x_2, y_1) * y_2 + y_1 * G_1 (x_1, x_2, y_2)\\
&+ [y_1, y_2, F_1 (x_1, x_2)] - (y_1 * y_2)*F_1 (x_1,x_2)
- F_1 (y_1, y_2)*(x_1 * x_2)\\
&+ F_1 ([x_1, x_2, y_1], y_2) + F_1 (y_1, [x_1, x_2, y_2]) + G_1 (y_1, y_2, x_1 * x_2) - F_1 (y_1 * y_2 , x_1 * x_2),
\end{aligned}
\end{equation}
\begin{equation}\nonumber
 \begin{aligned}
&G_0 (x_1, x_2, G_1(y_1, y_2, y_3)) + G_1 (x_1, x_2, G_0(y_1, y_2, y_3))\\
&= G_0 (G_1 (x_1, x_2, y_1), y_2, y_3) + G_0 (y_1, G_1 (x_1, x_2, y_2), y_3) + G_0 (y_1, y_2, G_1 (x_1, x_2, y_3))\\
&+ G_1 (G_0 (x_1, x_2, y_1), y_2, y_3) + G_1 (y_1, G_0 (x_1, x_2, y_2), y_3) + G_1 (y_1, y_2, G_0 (x_1, x_2, y_3))
\end{aligned}
\end{equation}
i.e.
\begin{equation}\label{E:3.23}
\begin{aligned}
&[x_1, x_2, G_1 (y_1, y_2, y_3)] + G_1 (x_1, x_2, [y_1, y_2, y_3])\\
&= [G_1 (x_1, x_2, y_1), y_2, y_3] + [y_1, G_1 (x_1, x_2, y_2), y_3]\\
&+ [y_1, y_2, G_1 (x_1, x_2, y_3)] \\
&+ G_1 ([x_1, x_2, y_1], y_2, y_3) + G_1 (y_1, [x_1, x_2, y_2], y_3) + G_1 (y_1, y_2, [x_1, x_2, y_3]).
\end{aligned}
\end{equation}
From (\ref{E:3.21})-(\ref{E:3.23}) and reminding (CC1)-(CC3), it is now easy to see the infinitesimal $(F_1, G_1)$ is a
$(2,3)$-cocycle with respect to the adjoint representation.
\par
For $k=2$, (\ref{E:3.18})-(\ref{E:3.20}) lead to
\begin{equation}\label{E:3.24}
 \begin{aligned}
  G_1(x_1, x_2, F_1 (y_1, y_2)) &= F_1 (G_1 (x_1, x_2, y_1), y_2) + F_1 (y_1, G_1 (x_1, x_2, y_2))\\
  &+ G_1 (y_1, y_2, F_1 (x_1, x_2))
  - F_1 (y_1, y_2)*F_1 (x_1, x_2)\\
  &- F_1 (y_1 * y_2, F_1 (x_1, x_2))
  - F_1 (F_1 (y_1, y_2), x_1 * x_2),
 \end{aligned}
\end{equation}
\begin{equation}\label{E:3.25}
\begin{aligned}
 G_1 (x_1, x_2, G_1(y_1, y_2, y_3)) &= G_1 (G_1 (x_1, x_2, y_1), y_2, y_3) + G_1 (y_1, G_1 (x_1, x_2, y_2), y_3)\\
 &+ G_1 (y_1, y_2, G_1 (x_1, x_2, y_3)).
 \end{aligned}
\end{equation}
Likewise, for $k=3$, (\ref{E:3.18})-(\ref{E:3.20}) lead to
\begin{equation}\label{E:3.26}
o(3) = 0,
\end{equation}
where $o(3):= F_1 (F_1 (y_1, y_2),F_1 (x_1, x_2))$. One observes that the equations (\ref{E:3.16}), (\ref{E:3.17}), (\ref{E:3.24}),
(\ref{E:3.21}), and (\ref{E:3.25}) imply that the
$(2,3)$-cocycle $(F_1, G_1)$ defines a Bol algebra of deformation type.
Therefore, by Proposition \ref{Pr3.3}, we get that the first order formal deformation $(f_t, g_t)$ of the Bol algebra $B$ generates a $t$-parameter infinitesimal deformation of $B$.
\begin{definition}\label{Def3.6}
Let $B$ be a Bol algebra. Let $(f_t, g_t)$ and $(f_t', g_t')$ be two $1$-parameter formal deformations of $B$, $f_t' := {\sum}_{i \geq 0} F_i' t^i$ and $g_t' := {\sum}_{i \geq 0} G_i' t^i$, where $F_0' = * = F_0$ and $G_0' = [ , , ] = G_0$. Then $(f_t', g_t')$ is said to be equivalent to $(f_t, g_t)$ (denote it by $(f_t', g_t') \thicksim (f_t', g_t')$) if there exists a formal isomorphism ${\Phi}_{t} := B_T' \rightarrow B_T $, ${\Phi}_{t} (x) := {\sum}_{i \geq 0} {\phi}_i (x) t^i$, where all ${\phi}_i$ are $\mathbb{K}$-linear maps $B \rightarrow B$ extended to be
$\mathbb{K} [[t]]$-linear, ${\phi}_0 = id$, such that, for all $x_i \in B$,
\par
$f_t' (x_1, x_2) = {\Phi}_{t}^{-1} f_t ({\Phi}_{t} (x_1),
{\Phi}_{t} (x_2))$,
\par
$g_t' (x_1, x_2, x_3) = {\Phi}_{t}^{-1} g_t ({\Phi}_{t} (x_1), {\Phi}_{t} (x_2), {\Phi}_{t} (x_3))$.\\
The deformation $(f_t, g_t)$ is called the null deformation (denoted by $(f_0, g_0)$) whenever $0 = (F_1, G_1) = (F_2, G_2) = \cdots$. The deformation $(f_t, g_t)$ is called a trivial deformation when $(f_t, g_t) \thicksim (f_0, g_0)$.
\end{definition}
Now the condition ${\Phi}_{t} (f_t' (x_1, x_2)) = f_t ({\Phi}_{t} (x_1), {\Phi}_{t} (x_2))$ is written as
\par
${\phi}_0 ({\sum}_{j \geq 0} F_j' (x_1, x_2) t^j)t^0 +
{\phi}_1 ({\sum}_{j \geq 0} F_j' (x_1, x_2) t^j)t^1 + \cdots$\par
$= F_0 ({\sum}_{l \geq 0} {\phi}_l (x_1)t^l , {\sum}_{l \geq 0} {\phi}_l (x_2)t^l) t^0 + F_1 ({\sum}_{l \geq 0} {\phi}_l (x_1)t^l , {\sum}_{l \geq 0} {\phi}_l (x_2)t^l) t^1+\cdots$\\
i.e., since ${\phi}_0 = id$ and $F_0 = *$,
\par
$F_0' (x_1, x_2)t^0 + F_1' (x_1, x_2)t^1 + \cdots +{\phi}_1 (F_0' (x_1, x_2))t^{0+1} + {\phi}_1 (F_1' (x_1, x_2))t^{1+1} + \cdots$
\par
$= ({\phi}_0 (x_1)t^0 + {\phi}_1 (x_1)t^1 + \cdots) *
({\phi}_0 (x_2)t^0 + {\phi}_1 (x_2)t^1 + \cdots)$
\par
$+ F_1 ({\phi}_0 (x_1)t^0 , {\phi}_0 (x_2)t^0 + {\phi}_1 (x_2)t^1 + \cdots)t^1 + F_1 ({\phi}_1 (x_1)t^1 , {\phi}_0 (x_2)t^0 + {\phi}_1 (x_2)t^1 + \cdots)t^1 + \cdots$.\\
Comparing the coefficients of $t^1$ from both sides, we get (having in mind ${\phi}_0 = id$ and $F_0' (x_1, x_2) = x_1 * x_2$)
\par
$F_1' (x_1, x_2) + {\phi}_1 (x_1 * x_2) = x_1 * {\phi}_1 (x_2) + {\phi}_1 (x_1) * x_2 + F_1 (x_1, x_2)$\\
i.e.
\begin{equation}\label{E:3.27}
 (F_1' - F_1) (x_1, x_2) = x_1 * {\phi}_1 (x_2) - x_2 * {\phi}_1 (x_1) - {\phi}_1 (x_1 * x_2).
\end{equation}
Likewise, proceeding as above, one finds out that the condition ${\Phi}_{t} (g_t' (x_1, x_2, x_3)) = \\ g_t ({\Phi}_{t} (x_1), {\Phi}_{t} (x_2), {\Phi}_{t} (x_3))$ gives
\begin{equation}\label{E:3.28}
 (G_1' - G_1) (x_1, x_2, x_3) = [x_1, x_2, {\phi}_1 (x_3)]
 + [{\phi}_1 (x_1), x_2, x_3] - [{\phi}_1 (x_2), x_1, x_3]
 - {\phi}_1 ([x_1, x_2, x_3]).
\end{equation}
Thus, by (\ref{E:3.27}) and (\ref{E:3.28}), we proved the following result.
\begin{theorem}\label{Th3.7}
Let $(f_t, g_t)$ and $(f_t', g_t')$ be two equivalent $1$-parameter formal deformations of a Bol algebra $B$. Then the infinitesimals $(F_1, G_1)$ and $(F_1' , G_1')$ of $(f_t, g_t)$ and $(f_t', g_t')$ respectively belong to the same cohomology class in $H^2 (B,B) \times H^3 (B,B)$.
\end{theorem}
One observes that the coboundary $(F_1' - F_1, G_1' - G_1)$ has companion $\eta \in B$ such that ${\Delta}_{D, \rho} (x_1,x_2)(\eta) = 0$,
$\forall x_i \in B$.
\section{Abelian extensions.}
In this section we aim to study abelian extensions of Bol algebras. We prove that any such an extension gives rise
to a representation and a $(2,3)$-cocycle of a given Bol algebra; moreover abelian extensions are classified by the $(2,3)$-cohomology groups.
\begin{definition}\label{Def4.1}
Let $(B, *, [ , , ])$, $(V, *_V, [ , , ]_V)$, and
$(\widehat{B}, *_{\widehat{B}}, [ , , ]_{\widehat{B}})$ be Bol algebras, $i: V \rightarrow \widehat{B}$ and
$p: \widehat{B} \rightarrow B$ be homomorphisms. The algebra $\widehat{B}$ is called an extension of $B$ by $V$
if the short sequence $E_{\widehat{B}}: 0 \rightarrow V \stackrel{i}{\rightarrow} \widehat{B} 
\stackrel{p}{\rightarrow} B \rightarrow 0$ is exact.
It is called an abelian extension of $B$ if $V$ is an abelian ideal of $\widehat{B}$ i.e. $u *_{\widehat{B}} v = 0$,
$[u,v, \cdot ]_{\widehat{B}} = 0$, $[u,\cdot , v]_{\widehat{B}} = 0$ and $[\cdot ,u, v]_{\widehat{B}} = 0$ for 
all $u, v \in V$. A section of $p$ is a linear map $\sigma : B \rightarrow {\widehat{B}}$ such that 
$p \circ \sigma = id_B$, in which case $E_{\widehat{B}}$ is said to be split.
\end{definition}
\begin{definition}\label{Def4.2}
Two extensions of Bol algebras $E_{\widehat{B}}: 0 \rightarrow V
\stackrel{i}{\rightarrow} \widehat{B} \stackrel{p}{\rightarrow} B \rightarrow 0$ and
$E_{\widetilde{B}}: 0 \rightarrow V \stackrel{j}{\rightarrow} \widetilde{B} 
\stackrel{q}{\rightarrow} B \rightarrow 0$ are said to be equivalent if there is a Bol algebra homomorphism
$\varphi : \widehat{B} \rightarrow \widetilde{B}$ such that the diagram
\[
 \begin{array}{clllr}
 0 \longrightarrow & V  \stackrel{i}{\longrightarrow} &  \widehat{B} \stackrel{p}{\longrightarrow}& B &\longrightarrow 0 \\
                   & \downarrow{id}                   &  \downarrow{\varphi}                      & \downarrow{id}   \\
 0 \longrightarrow & V \stackrel{j}{\longrightarrow}  &  \widetilde{B} \stackrel{q}{\longrightarrow}& B &\longrightarrow 0
 \end{array}
\]
commutes.
\end{definition}
The following result provides a construction method of a representation of a given Bol algebra $B$ from an abelian
extension of $B$.
\begin{proposition}\label{Pr4.3}
Let $V$ be a vector space viewed as a Bol algebra with trivial operations and let
$E_{\widehat{B}}: 0 \rightarrow V \stackrel{i}{\rightarrow} \widehat{B} \stackrel{p}{\rightarrow} B \rightarrow 0$
be an abelian split extension of a Bol algebra $B$. Define the maps $\rho : B \rightarrow End (V)$,
$D, \theta : B \times B \rightarrow End (V)$ by
\begin{equation}\label{E:4.1}
 \rho (x) (u):= \sigma (x) *_{\widehat{B}} u, 
\end{equation}
\begin{equation}\label{E:4.2}
 D(x_1, x_2) (u) := [\sigma (x_1), \sigma (x_2), u]_{\widehat{B}}, 
\end{equation}
\begin{equation}\label{E:4.3}
 \theta (x_1, x_2) (u) := [u, \sigma (x_1), \sigma (x_2)]_{\widehat{B}}, 
\end{equation}
for all $x, x_i \in B$ and $u \in V$. Then
\par
(i) $(\rho, D, \theta)$ is a representation of $B$ in $V$;
\par
(ii) the representation $(\rho, D, \theta)$ does not depend on the choice of the section $\sigma$;
\par
(iii) equivalent abelian split extensions of $B$ give the same representation of $B$.
\end{proposition}
\begin{proof}
 (i) We need to check (R1), (R21), (R22), (R31), (R32), and (R33) for $(\rho, D, \theta)$ as defined by
(\ref{E:4.1})-(\ref{E:4.3}). 
\par 
By (B1) in ${\widehat{B}}$ we have
\par
$[\sigma (x_1), \sigma (x_2), u]_{\widehat{B}} + [\sigma (x_1), u, \sigma (x_2)]_{\widehat{B}}
+ [u, \sigma (x_1), \sigma (x_2)]_{\widehat{B}} = 0$\\
i.e., using (\ref{E:4.2}) and (\ref{E:4.3}),
\par
$D(x_1, x_2)(u) - \theta (x_2, x_1)(u) + \theta (x_1, x_2)(u) = 0$\\
so that we get (R1).
\par
(By (B2) in ${\widehat{B}}$ we have
\par
$[\sigma (x_1), \sigma (x_2), \sigma (y) *_{\widehat{B}} u]_{\widehat{B}}
= [\sigma (x_1), \sigma (x_2), \sigma (y)]_{\widehat{B}} *_{\widehat{B}} u 
+ \sigma (y) *_{\widehat{B}} [\sigma (x_1), \sigma (x_2), u]_{\widehat{B}}$
\par
\hspace{4.0cm}$+ [\sigma (y), u, \sigma (x_1) *_{\widehat{B}} \sigma (x_2)]_{\widehat{B}}
- (\sigma (y) *_{\widehat{B}} u) *_{\widehat{B}} (\sigma (x_1) *_{\widehat{B}} \sigma (x_2))$\\
i.e., using (\ref{E:4.1}) and (\ref{E:4.2}),
\par
$[\sigma (x_1), \sigma (x_2), \rho (y)(u)]_{\widehat{B}} = [\sigma (x_1), \sigma (x_2), \sigma (y)]_{\widehat{B}} *_{\widehat{B}} u
+ \sigma (y) *_{\widehat{B}} D(x_1, x_2)(u)$
\par
\hspace{3.7cm}$- [u, \sigma (y), \sigma (x_1) *_{\widehat{B}} \sigma (x_2)]_{\widehat{B}}
- \rho (y)(u) *_{\widehat{B}} (\sigma (x_1) *_{\widehat{B}} \sigma (x_2))$\\
or
\par
$\big{(} D(x_1, x_2) \circ \rho (y) \big{)}(u) = \rho ([x_1, x_2, y])(u) + \big{(} \rho (y) \circ D(x_1, x_2) \big{)}(u)$
\par
\hspace{3.7cm}$- \theta (y, x_1 * x_2)(u) + \rho (x_1 * x_2) (\rho (y)(u))$\\
(here we used the fact that $[\sigma (x_1), \sigma (x_2), \sigma (y)]_{\widehat{B}} - \sigma ([x_1, x_2, y])
\in Ker (p) = Im (i) \cong V$ and $\sigma (x_1) *_{\widehat{B}} \sigma (x_2) - \sigma (x_1 * x_2) \in Ker (p) = Im (i) \cong V$,
and next we applied (\ref{E:4.3})). Thus we obtain
\par
$\big{(} D(x_1, x_2) \circ \rho (y)  -  \rho (y) \circ D(x_1, x_2) \big{)} (u) = \rho ([x_1, x_2, y])(u)
- \theta (y, x_1 * x_2) (u)$
\par
\hspace{6.5cm}$+ \big{(} \rho (x_1 * x_2) \circ \rho (y) \big{)} (u)$\\
which is (R21).
\par
Again by (B2) in ${\widehat{B}}$, we have
\par
$[\sigma (x), u, \sigma (y_1) *_{\widehat{B}} \sigma (y_2) ]_{\widehat{B}} = 
[\sigma (x), u, \sigma (y_1)]_{\widehat{B}} *_{\widehat{B}} \sigma (y_2) 
+ \sigma (y_1) *_{\widehat{B}} [\sigma (x), u, \sigma (y_2)]_{\widehat{B}}$
\par
\hspace{4.0cm}$+ [\sigma (y_1), \sigma (y_2), \sigma (x) *_{\widehat{B}} u]_{\widehat{B}} 
- (\sigma (y_1) *_{\widehat{B}} \sigma (y_2)) *_{\widehat{B}} (\sigma (x) *_{\widehat{B}} u)$\\
which leads to (R22) (proceeding as above).
\par
Likewise, in ${\widehat{B}}$, the equality (see (B3))
\par
$[\sigma (x_1), \sigma (x_2), [\sigma (y_1), \sigma (y_2), u]_{\widehat{B}}]_{\widehat{B}}
= [[\sigma (x_1), \sigma (x_2), \sigma (y_1)]_{\widehat{B}}, \sigma (y_2), u]_{\widehat{B}}$
\par
\hspace{5.5cm}$+ [\sigma (y_1), [\sigma (x_1), \sigma (x_2), \sigma (y_2)]_{\widehat{B}}, u]_{\widehat{B}}$
\par
\hspace{5.5cm}$+ [\sigma (y_1), \sigma (y_2), [\sigma (x_1), \sigma (x_2), u]_{\widehat{B}}]_{\widehat{B}}$\\
implies (R31), the equality
\par
$[\sigma (x_1), \sigma (x_2), [u, \sigma (y_1), \sigma (y_2)]_{\widehat{B}}]_{\widehat{B}}
= [[\sigma (x_1), \sigma (x_2), u]_{\widehat{B}}, \sigma (y_1), \sigma (y_2)]_{\widehat{B}}$
\par
\hspace{5.5cm}$+ [u, [\sigma (x_1), \sigma (x_2), \sigma (y_1)]_{\widehat{B}}, \sigma (y_2)]_{\widehat{B}}$
\par
\hspace{5.5cm}$+ [u, \sigma (y_1), [\sigma (x_1), \sigma (x_2), \sigma (y_2)]_{\widehat{B}}]_{\widehat{B}}$\\
implies (R32), and the equality
\par
$[u, \sigma (x), [\sigma (y_1), \sigma (y_2), \sigma (y_3)]_{\widehat{B}}]_{\widehat{B}}
= [[u, \sigma (x), \sigma (y_1)]_{\widehat{B}}, \sigma (y_2), \sigma (y_3)]_{\widehat{B}}$
\par
\hspace{5.5cm}$+ [\sigma (y_1), [u, \sigma (x), \sigma (y_2)]_{\widehat{B}}, \sigma (y_3)]_{\widehat{B}}$
\par
\hspace{5.5cm}$+ [\sigma (y_1), \sigma (y_2), [u, \sigma (x), \sigma (y_3)]_{\widehat{B}}]_{\widehat{B}}$\\
implies (R33). 
\par
Thus we get that $(\rho, D, \theta)$ is a representation of $B$ in $V$.
\par
(ii) We proceed to show that the representation $(\rho, D, \theta)$ of $B$  does not depend on the choice of the 
section $\sigma : B \rightarrow \widehat{B}$.
\par
Suppose that ${\sigma}' : B \rightarrow \widehat{B}$ is another section. Then, $\forall x_i \in B$, 
$p ({\sigma}' (x_i) - \sigma (x_i)) = 0$ which implies that ${\sigma}' (x_i) - \sigma (x_i) \in V$ so that
${\sigma}' (x_i) = \sigma (x_i) + v_i$ for some $v_i \in V$. Next, $\forall u \in V$, we have
\par
${\sigma}' (x_i) *_{\widehat{B}} u = (\sigma (x_i) + v_i) *_{\widehat{B}} u = \rho (x_i) (u)$;
\par
$[{\sigma}' (x_1), {\sigma}' (x_2), u]_{\widehat{B}} = [\sigma (x_1) + v_1, \sigma (x_2) + v_2, u]_{\widehat{B}}
= D(x_1,x_2)(u)$;
\par
$[u, {\sigma}' (x_1), {\sigma}' (x_2) ]_{\widehat{B}} = [u, \sigma (x_1) + v_1, \sigma (x_2) + v_2 ]_{\widehat{B}}
= \theta (x_1, x_2)(u) $.\\
Thus $(\rho, D, \theta)$ does not depend on the choice of $\sigma$.
\par
(iii) Suppose that $E_{\widehat{B}}$ and $E_{\widetilde{B}}$ are equivalent abelian split extensions of $B$, i.e. there
is a Bol algebra homomorphism $\varphi : {\widehat{B}} \rightarrow {\widetilde{B}}$ satisfying $\varphi \circ i = j$
and $q \circ \varphi = p$ (see Definition \ref{Def4.2}). Let $\sigma$ and ${\sigma}'$ some sections of $p$ and $q$
respectively. We have, $\forall x \in B$, $(q \circ \varphi \circ \sigma) (x) = (p \circ \sigma)(x) =x= 
(q \circ {\sigma}')(x)$. Thus $(\varphi \circ \sigma) (x) - {\sigma}' (x) \in Ker (q) \cong V$ and so there is a
$v \in V$ such that ${\sigma}' (x) = (\varphi \circ \sigma) (x) + v$. Therefore, $\forall u \in V$, 
\par
${\sigma}' (x) *_{\widetilde{B}} u =((\varphi \circ \sigma) (x) + v) *_{\widetilde{B}} u =\sigma (x) *_{\widehat{B}} u$\\
(in  the last equality we used the fact that $i(V) \cong V \cong j(V) \cong \varphi (V)$) and so equivalent abelian
split extensions give the same $\rho$.
\par
Likewise one gets that equivalent abelian split extensions give the same $D$ and $\theta$. This completes the proof.
\end{proof}
The following result shows that any abelian split extension of a Bol algebra $B$ induces a certain $(2,3)$-cocycle
of $B$.
\begin{proposition}\label{Pr4.4}
Let $V$ be a vector space viewed as a Bol algebra with trivial operations, and let
$0 \rightarrow V \rightarrow \widehat{B} \rightarrow B \rightarrow 0$ be an abelian split extension of the Bol
algebra $B$ by $V$. For all $x_i \in B$, define the maps
\begin{equation}\label{E:4.4}
 \nu (x_1, x_2) := \sigma (x_1) *_{\widehat{B}} \sigma (x_2) - \sigma (x_1 * x_2), 
\end{equation}
\begin{equation}\label{E:4.5}
\omega (x_1, x_2, x_3) := [\sigma (x_1), \sigma (x_2), \sigma (x_3)]_{\widehat{B}} - \sigma ([x_1, x_2, x_3]). 
\end{equation}
Then $(\nu, \omega)$ is a $(2,3)$-cocycle of $B$ with coefficients in $V$, with respect to the representation
$(\rho,D, \theta)$ given by (\ref{E:4.1})-(\ref{E:4.3}).
\end{proposition}
\begin{proof}
 Obviously $\nu \in C^2 (B,V)$ and $\omega \in C^3 (B,V)$. Next we must check that $(\nu , \omega)$ satisfies
 (CC1)-(CC3) (see Definition \ref{Def2.11}).
 \par
 That $\omega$ satisfies (CC1) and (CC3) is proved by Lemma 4 in \cite{Zh}. Thus we are done if we check (CC2) for 
 $(\nu , \omega)$. Consider the equality (see (B2))
 \par
 $[\sigma (x_1), \sigma (x_2), \sigma (y_1) *_{\widehat{B}} \sigma (y_2)]_{\widehat{B}}
 = [\sigma (x_1), \sigma (x_2), \sigma (y_1)]_{\widehat{B}} *_{\widehat{B}} \sigma (y_2)$
 \par
 \hspace{5.0cm}$+ \sigma (y_1) *_{\widehat{B}} [\sigma (x_1), \sigma (x_2), \sigma (y_2)]_{\widehat{B}}$
 \par
 \hspace{5.0cm}$+ [\sigma (y_1), \sigma (y_2), \sigma (x_1) *_{\widehat{B}} \sigma (x_2)]_{\widehat{B}}$
 \par
 \hspace{5.0cm}$- (\sigma (y_1) *_{\widehat{B}} \sigma (y_2)) *_{\widehat{B}} (\sigma (x_1) *_{\widehat{B}} \sigma (x_2))$,\\
 then we obtain that its left-hand side is
 \par
 $[\sigma (x_1), \sigma (x_2), \sigma (y_1) *_{\widehat{B}} \sigma (y_2)]_{\widehat{B}}
 = [\sigma (x_1), \sigma (x_2), \nu (y_1, y_2) + \sigma (y_1 * y_2) ]_{\widehat{B}}$ (by (\ref{E:4.4}))
 \par
 \hspace{4.7cm}$= D(x_1, x_2) \nu (y_1, y_2) + [\sigma (x_1), \sigma (x_2), \sigma (y_1 * y_2)]_{\widehat{B}}$
 (by (\ref{E:4.2}))
 \par
 \hspace{4.7cm}$= D(x_1, x_2) \nu (y_1, y_2) + \omega (x_1, x_2, y_1 *  y_2) + \sigma ([x_1, x_2, y_1 * y_2])$ (by (\ref{E:4.5}))\\
 while its right-hand side is 
 \par
 $[\sigma (x_1), \sigma (x_2), \sigma (y_1)]_{\widehat{B}} *_{\widehat{B}} \sigma (y_2)
 + \sigma (y_1) *_{\widehat{B}} [\sigma (x_1), \sigma (x_2), \sigma (y_2)]_{\widehat{B}}$
 \par
 $+ [\sigma (y_1), \sigma (y_2), \sigma (x_1) *_{\widehat{B}} \sigma (x_2)]_{\widehat{B}}
 - (\sigma (y_1) *_{\widehat{B}} \sigma (y_2)) *_{\widehat{B}} (\sigma (x_1) *_{\widehat{B}} \sigma (x_2))$
 \par
 $= (\omega (x_1, x_2, y_1) + \sigma ([x_1, x_2, y_1])) *_{\widehat{B}} \sigma (y_2)
 + \sigma (y_1) *_{\widehat{B}} (\omega (x_1, x_2, y_2) + \sigma ([x_1, x_2, y_2]))$
 \par
 $+ [\sigma (y_1), \sigma (y_2), \nu (x_1, x_2) + \sigma (x_1 * x_2) ]_{\widehat{B}}$
 \par
 $- (\nu (y_1, y_2) + \sigma (y_1 * y_2)) *_{\widehat{B}} (\nu (x_1, x_2) + \sigma (x_1 * x_2))$
 (by (\ref{E:4.4}) and (\ref{E:4.5}))
 \par
 $= - \rho (y_2) \omega (x_1, x_2, y_1) + \nu ([x_1, x_2, y_1], y_2) + \sigma ([x_1, x_2, y_1] * y_2)$
 \par
 $+ \rho (y_1) \omega (x_1, x_2, y_2) + \nu (y_1, [x_1, x_2, y_2]) + \sigma (y_1 * [x_1, x_2, y_2])$
 \par
 $+ D(y_1, y_2) \nu (x_1, x_2) + \omega (y_1, y_2, x_1 *  x_2) + \sigma ([y_1, y_2, x_1 * x_2])$
 \par
 $+  \rho (x_1 * x_2) \nu (y_1, y_2) - \rho (y_1 * y_2) \nu (x_1, x_2) - \nu (y_1 * y_2, x_1 * x_2)
 - \sigma ((y_1 * y_2) * (x_1 * x_2))$. \\
 Therefore, comparing these expressions of the left-hand and right-hand sides and using (B2) for $B$, we have
 \par
 $ D(x_1, x_2) \nu (y_1, y_2) + \omega (x_1, x_2, y_1 *  y_2) = \nu ([x_1, x_2, y_1], y_2) + \nu (y_1, [x_1, x_2, y_2])$
 \par
 \hspace{5.8cm}$+  D(y_1, y_2) \nu (x_1, x_2) + \omega (y_1, y_2, x_1 *  x_2)$
 \par
 \hspace{5.8cm}$- \rho (y_2) \omega (x_1, x_2, y_1) + \rho (y_1) \omega (x_1, x_2, y_2)$
 \par
 \hspace{5.8cm}$+ \rho (x_1 * x_2) \nu (y_1, y_2) - \rho (y_1 * y_2) \nu (x_1, x_2)$
 \par
 \hspace{5.8cm}$- \nu (y_1 * y_2, x_1 * x_2)$\\
 which is (CC2). This completes the proof.
\end{proof}
Given a Bol algebra $B$, $V$ its $B$-module, and $(\nu, \omega)$ a $(2,3)$-cocycle of $B$, then Theorem \ref{Th3.4} said
that $B {\oplus}_{(\nu, \omega)} V := (B \oplus V, *_{\nu}, [ , , ]_{\omega})$ is a Bol algebra. Therefore, the short
exact sequence $0 \rightarrow V \rightarrow B {\oplus}_{(\nu, \omega)} V \rightarrow B \rightarrow 0$ is an extension
of $B$ by $V$.
\begin{proposition}\label{Pr4.5}
Let $E_{(\nu, \omega)} : 0 \rightarrow V \rightarrow B {\oplus}_{(\nu, \omega)} V \rightarrow B \rightarrow 0$
and  $E_{(\nu ', \omega ')} : 0 \rightarrow V \rightarrow B {\oplus}_{(\nu ', \omega ')} V \rightarrow B \rightarrow 0$
be two abelian extensions of a Bol algebra $B$ by its $B$-module $V$. Then $E_{(\nu, \omega)}$ and $E_{(\nu ', \omega ')}$ are equivalent if and only if $(\nu, \omega)$ and $(\nu ', \omega ')$ belong to the same cohomology class.
\end{proposition}
\begin{proof}
 Let $\varphi : B {\oplus}_{(\nu, \omega)} V  \rightarrow B {\oplus}_{(\nu ', \omega ')} V$ be the corresponding
 homomorphism, i.e.
 \begin{equation}\label{E:4.6}
  \varphi ((x_1 + u_1) *_{\nu} (x_2 + u_2)) = \varphi (x_1 + u_1) *_{\nu '} \varphi (x_2 + u_2)
 \end{equation}
\begin{equation}\label{E:4.7}
 \varphi ([x_1 + u_1, x_2 + u_2, x_3 + u_3]_{\omega}) = [\varphi (x_1 + u_1), \varphi (x_2 + u_2), \varphi (x_3 + u_3)]_{\omega '}
\end{equation}
where $x_i \in B$ and $u_i \in V$. Then the commutativity of diagram (see Definition \ref{Def4.2}) implies that there is a
linear map $f : B \rightarrow V$ such that $\varphi (x_i + u_i) = x_i + f(x_i) + u_i$. The left-hand side of
(\ref{E:4.6}) is
\par
$\varphi ((x_1 + u_1) *_{\nu} (x_2 + u_2)) = \varphi (x_1 * x_2 + \nu (x_1, x_2) + \rho (x_1) (u_2) - \rho (x_2) (u_1))$ (by (\ref{E:3.9}))
\par
\hspace{1.8cm}$= x_1 * x_2 + f(x_1 * x_2) + \nu (x_1, x_2) + \rho (x_1) (u_2) - \rho (x_2) (u_1)$\\
while its right-hand side is
\par
$\varphi (x_1 + u_1) *_{\nu '} \varphi (x_2 +u_2) = x_1 * x_2 + \nu ' (x_1, x_2) + \rho (x_1) (f(x_2) + u_2) - \rho (x_2) (f(x_1) + u_1)$. \\
Thus we obtain
\begin{equation}\label{E:4.8}
 (\nu - \nu ') (x_1, x_2) = \rho (x_1) f(x_2) - \rho (x_2) f(x_1) - f(x_1 * x_2).
\end{equation}
Likewise, from (\ref{E:4.7}), we obtain
\begin{equation}\label{E:4.9}
 (\omega - \omega ') (x_1, x_2, x_3) = D(x_1, x_2)f(x_3) - \theta (x_1, x_3)f(x_2) + \theta (x_2, x_3)f(x_1) - f([x_1, x_2, x_3]).
\end{equation}
Therefore, (\ref{E:4.8}) and (\ref{E:4.9}) imply that $(\nu, \omega)$ and $(\nu ', \omega ')$ are in the same
cohomology class of the cohomology group $H^2 (B,V) \times H^3 (B,V)$.
\par
Conversely, let $(\nu - \nu ', \omega - \omega ') \in B^2 (B,V) \times B^3 (B,V)$ and has companion $\eta$. 
Then there is a linear map $f : B \rightarrow V$ such that
\par
$(\nu - \nu ') (x_1, x_2) = \rho (x_1) f(x_2) - \rho (x_2) f(x_1) + {\Delta}_{D,\rho} (x_1,x_2) (\eta) - f(x_1 * x_2)$
\par
$(\omega - \omega ')(x_1, x_2, x_3) = D(x_1, x_2)f(x_3) - \theta (x_1, x_3)f(x_2) + \theta (x_2, x_3)f(x_1) - f([x_1, x_2, x_3])$.\\
Now define a map $\varphi : B {\oplus}_{(\nu, \omega)} V  \rightarrow B {\oplus}_{(\nu ', \omega ')} V$ by setting
$\varphi (x+u) := x + \widetilde {f}(x) +u$, where $\widetilde{f} = f + \widehat{f}$ and $\widehat{f}$ is a pseudoderivation of $B$ with companion $- \eta$ i.e.
\par
$\widehat{f} (x_1 * x_2) = \rho (x_1) \widehat{f}(x_2) - \rho (x_2) \widehat{f}(x_1) + {\Delta}_{D,\rho} (x_1,x_2) (- \eta) - \widehat{f}(x_1 * x_2)$
\par
\hspace{1.7cm}$= \rho (x_1) \widehat{f}(x_2) - \rho (x_2) \widehat{f}(x_1) - {\Delta}_{D,\rho} (x_1,x_2) (\eta) - \widehat{f}(x_1 * x_2)$,
\par
$\widehat{f} ([x_1, x_2, x_3]) = \theta (x_2, x_3)\widehat{f}(x_1) - \theta (x_1, x_3)\widehat{f}(x_2) + D(x_1, x_2)\widehat{f}(x_3)$.\\
Clearly $\widetilde{f}$ is a linear map $B \rightarrow V$ and $\varphi$ is well-defined. We proceed to show that $\varphi$ is a homomorphism. We have
\par
$\varphi ((x_1 + u_1) *_{\nu} (x_2 + u_2)) = \varphi (x_1 * x_2 + \nu (x_1, x_2) + \rho (x_1)(u_2) - \rho (x_2)(u_1))$ (by (\ref{E:3.9})
\par
\hspace{4.0cm}$= x_1 * x_2 + (f + \widehat{f})(x_1 * x_2) + \nu (x_1, x_2) + \rho (x_1)(u_2) - \rho (x_2)(u_1)$,
\par
$\varphi (x_1 + u_1) *_{\nu '} \varphi (x_2 + u_2) = x_1 * x_2 + \nu ' (x_1, x_2) + \rho (x_1) (\widetilde{f}(x_2)+u_2) - \rho (x_2) (\widetilde{f}(x_1)+u_1)$
\par
\hspace{4.0cm}$= x_1 * x_2 + \nu ' (x_1, x_2) + \rho (x_1) (f + \widehat{f})(x_2) +  \rho (x_1)(u_2)$
\par
\hspace{4.0cm}$-\rho (x_2) (f + \widehat{f})(x_1)-  \rho (x_2)(u_1)$.\\
Therefore
\par
$\varphi ((x_1 + u_1) *_{\nu} (x_2 + u_2)) - \varphi (x_1 + u_1) *_{\nu '} \varphi (x_2 + u_2)$
\par
$= (\nu - \nu ') (x_1, x_2) + (f + \widehat{f})(x_1 * x_2) - \rho (x_1) (f + \widehat{f})(x_2) + \rho (x_2) (f + \widehat{f})(x_1)$
\par
$= (\nu - \nu ') (x_1, x_2) + f(x_1 * x_2) - \rho (x_1) f(x_2) + \rho (x_2) f(x_1) + \widehat{f}(x_1 * x_2) - \rho (x_1) \widehat{f}(x_2) + \rho (x_2) \widehat{f}(x_1) $
\par
$= \rho (x_1) f(x_2) - \rho (x_2) f(x_1) + {\Delta}_{D,\rho} (x_1,x_2) (\eta) - f(x_1 * x_2) + f(x_1 * x_2) - \rho (x_1) f(x_2)$
\par
$+ \rho (x_2) f(x_1) - {\Delta}_{D,\rho} (x_1,x_2) (\eta)$ (by the expressions of $\nu - \nu '$ and $\widehat{f}$)
\par
$= 0$\\
so we get that
\begin{equation}\label{E:4.10}
 \varphi ((x_1 + u_1) *_{\nu} (x_2 + u_2)) = \varphi (x_1 + u_1) *_{\nu '} \varphi (x_2 + u_2).
\end{equation}
Next we have
\par
$\varphi ([x_1 + u_1, x_2 + u_2, x_3 + u_3]_{\omega}) =
[x_1, x_2, x_3] + \widetilde{f} ([x_1, x_2, x_3]) +
\omega (x_1, x_2, x_3)$
\par
\hspace{5.0cm}$+ D(x_1, x_2)(u_3) - \theta (x_1, x_3)(u_2) + \theta (x_2, x_3)(u_1)$\\
while
\par
$[\varphi (x_1 + u_1), \varphi (x_2 + u_2), \varphi (x_3 + u_3)]_{\omega '} = [x_1, x_2, x_3] + {\omega '} (x_1,x_2,x_3)$
\par
$+ D(x_1,x_2)(\widetilde{f} (x_3) + u_3)
- \theta (x_1, x_3) (\widetilde{f} (x_2) + u_2)
+ \theta (x_2, x_3) (\widetilde{f} (x_1) + u_1)$
\par
$= [x_1, x_2, x_3] + {\omega '} (x_1,x_2,x_3) + D(x_1, x_2) f(x_3)$
\par
$+ D(x_1, x_2) \widehat{f}(x_3) + D(x_1,x_2) (u_3) - \theta (x_1, x_3) f(x_2) - \theta (x_1, x_3)
\widehat{f} (x_2) - \theta (x_1, x_3) (u_2)$
\par
$+ \theta (x_2, x_3) f(x_1) + \theta (x_2, x_3)
\widehat{f} (x_1) + \theta (x_2, x_3) (u_1)$.\\
Thus
\par
$\varphi ([x_1 + u_1, x_2 + u_2, x_3 + u_3]_{\omega}) -
[\varphi (x_1 + u_1), \varphi (x_2 + u_2), \varphi (x_3 + u_3)]_{\omega '}$
\par
$= f([x_1, x_2, x_3]) - \theta (x_2, x_3) f(x_1) + \theta (x_1, x_3) f(x_2) - D(x_1,x_2) f(x_3) + \widehat{f} ([x_1, x_2, x_3])$
\par
\hspace{0.5cm}$- \theta (x_2, x_3)
\widehat{f} (x_1) + \theta (x_1, x_3)
\widehat{f} (x_2) - D(x_1, x_2) \widehat{f}(x_3) + (\omega - \omega ') (x_1, x_2, x_3)$
\par
$= 0$ (by the definition of $\omega - \omega '$ and
$\widehat{f}$)\\
so we get
\begin{equation}\label{E:4.11}
 \varphi ([x_1 + u_1, x_2 + u_2, x_3 + u_3]_{\omega}) =
[\varphi (x_1 + u_1), \varphi (x_2 + u_2), \varphi (x_3 + u_3)]_{\omega '}.
\end{equation}
Therefore, by (\ref{E:4.10}) and (\ref{E:4.11}), we obtain that there is a homomorphism $\varphi : B {\oplus}_{(\nu, \omega)} V  \rightarrow B {\oplus}_{(\nu ', \omega ')} V$ so that $E_{(\nu, \omega)}$ and $E_{(\nu', \omega')}$ are equivalent. This completes the proof.
\end{proof}
\begin{remark}\label{Rem4.6}
From the consideration above (Propositions 4.3, 4.4, and 4.5) we get that there is a one-to-one correspondence between the set of equivalence classes of abelian extensions
of a given Bol algebra $B$ by its $B$-module $V$ and the cohomology group $H^2 (B,V) \times H^3 (B,V)$.
\end{remark}
\par
\par
\section*{Acknowledgement}
The author would like to thank Professor Tao Zhang (Henan Normal University, Xinxiang, P. R. China) for his observation on the earlier version of Proposition \ref{Pr4.5} which led to its correction as above. Thanks also go to the referee for useful suggestions that helped in the improvement of this paper.

\footnotesize{

\end{document}